\documentclass[letterpaper,11pt]{amsart}
\usepackage[margin=1in]{geometry}\usepackage{amsfonts}
\usepackage{pb-diagram}
\title[Right-angled Artin groups in mapping class groups]{Right-angled Artin groups and a generalized isomorphism problem for finitely generated
subgroups of mapping class groups}

\newtheorem{thm}{Theorem}
\newtheorem{lemma}[thm]{Lemma}

\newtheorem{conje}[thm]{Conjecture}

\newtheorem{cor}[thm]{Corollary}
\newtheorem{prop}[thm]{Proposition}

\newtheorem{quest}[thm]{Question}

\numberwithin{thm}{section}

\newcommand{\al}{\alpha}
\newcommand{\eps}{\epsilon}
\newcommand{\gam}{\Gamma}

\newcommand{\bZ}{\mathbb{Z}}
\newcommand{\bR}{\mathbb{R}}
\newcommand{\bC}{\mathbb{C}}

\DeclareMathOperator{\Aut}{Aut}
\DeclareMathOperator{\CAT}{CAT}
\DeclareMathOperator{\Out}{Out}
\DeclareMathOperator{\Homeo}{Homeo}
\DeclareMathOperator{\Mod}{Mod}
\DeclareMathOperator{\Isom}{Isom}

\DeclareMathOperator{\rk}{rk}

\DeclareMathOperator{\Inn}{Inn}

\DeclareMathOperator{\vcd}{vcd }
\DeclareMathOperator{\lk}{lk}
\DeclareMathOperator{\St}{st}
\newcommand{\yt}{\widetilde}

\newcommand{\bP}{\mathbb{P}}

\newcommand{\bQ}{\mathbb{Q}}

\newcommand{\bH}{\mathbb{H}}

\newcommand{\mF}{\mathcal{F}}
\newcommand{\mL}{\mathcal{L}}

\newcommand{\mC}{\mathcal{C}}

\author[T. Koberda]{Thomas Koberda}
\address{Department of Mathematics\\ Harvard University\\ 1 Oxford St.\\ Cambridge, MA 02138 }
\email{ koberda@math.harvard.edu}
\begin{document}
\begin{abstract}
Consider the mapping class group $\Mod_{g,p}$ of a surface $\Sigma_{g,p}$ of genus $g$ with $p$ punctures, and a finite collection $\{f_1,\ldots,f_k\}$ of mapping classes, each of which is either a Dehn twist about a simple closed curve or a pseudo-Anosov homeomorphism supported on a connected subsurface.  In this paper we prove that for all sufficiently large $N$, the mapping classes $\{f_1^N,\ldots,f_k^N\}$ generate a right-angled Artin group.  The right-angled Artin group which they generate can be determined from the combinatorial topology of the mapping classes themselves.  When $\{f_1,\ldots,f_k\}$ are arbitrary mapping classes, we show that sufficiently large powers of these mapping classes generate a group which embeds in a right-angled Artin group in a controlled way.  We establish some analogous results for real and complex hyperbolic manifolds.  We also discuss the unsolvability of the isomorphism problem for finitely generated subgroups of $\Mod_{g,p}$, and prove that the isomorphism problem for right-angled Artin groups is solvable.  We thus characterize the isomorphism type of many naturally occurring subgroups of $\Mod_{g,p}$.
\end{abstract}
\maketitle
\begin{center}
\today
\end{center}
\tableofcontents
\section{Introduction and statement of results}
\subsection{Overview}
The goal of this paper is to describe a large class of finitely generated subgroups of the mapping class group.  Several years ago, B. Farb conjectured that given any finite collection of mapping classes, sufficiently large powers of these classes generate a right-angled Artin group.  We resolve Farb's conjecture in this article.

The task of describing all finitely generated subgroups up to isomorphism is an impossible one, as we will show precisely.

\subsection{Right-angled Artin subgroups of mapping class groups}
Let $\Sigma=\Sigma_{g,p}$ be a surface of finite genus $g$ and a finite set $P$ of $p\geq 0$ punctures.  We will require that \[\chi(\Sigma)=2-2g-p<0.\]  Recall that the mapping class group $\Mod_{g,p}$ is defined by \[\Mod_{g,p}\cong\pi_0(\Homeo(\Sigma,P)),\] namely the group of orientation--preserving self--homeomorphisms of $\Sigma$ which preserve the set of punctures, up to isotopy.  It is sometimes important to distinguish between mapping classes which preserve $P$ pointwise and those which preserve $P$ setwise.  This distinction is not important for our considerations, as the former sits within the latter with finite index.

We first recall the notion of a right-angled Artin system and a right-angled Artin group.  Let $\gam$ be a finite graph with vertex set $V$ and edge set $E$.  The set of vertices of this graph together with their incidence (which is to say adjacency) relations determine a {\bf right-angled Artin system}.  This terminology is motivated Coxeter theory.  A {\bf right-angled Artin group} is a pair $(A(\gam),V)$, where $V$ is a right-angled Artin system and a group $A(\gam)$ with a presentation given by \[A(\gam)=\langle V\mid [v_i,v_j] \textrm{ whenever }(v_i,v_j)\in E\rangle.\]  We call the size of the vertex set of $\gam$ the {\bf rank} of $A(\gam)$, which is terminology motivated by the fact that this is the smallest number of generators for $A(\gam)$.  When the right-angled Artin system for $A(\gam)$ is the set of vertices of $\gam$, we call the vertices of $\gam$ {\bf vertex generators}.

For curves and connected subsurfaces of a given surface $\Sigma$, there is a notion of a right-angled Artin system as well.  Let $F$ be a nonempty, finite collection of Dehn twists or pseudo-Anosov homeomorphisms supported on connected subsurfaces of $\Sigma$.  The support $\Sigma_i$ of each element $f_i\in F$ is either a simple closed curve or a connected subsurface.   We can build a graph $\gam$ by taking the vertices of $\gam$ to be the elements of $F$ and connecting two vertices by an edge if and only if the supports can be realized disjointly.  This graph is called the {\bf coincidence correspondence of $F$}.  We define $A(F)$ to be $A(\gam)$.

Call a subset $F=\{f_1,\ldots,f_k\}\subset\Mod_{g,p}$ {\bf irredundant} if there are no linear relations between commuting subcollections of $F$.  Precisely, let $\{f_1,\ldots,f_i\}=F'\subset F$ be a collection of elements which generate an abelian subgroup of $\Mod_{g,p}$.  The elements $F'$ satisfy a {\bf linear relation} if (after writing the group law in this subgroup additively) we have a nontrivial solution to the equation \[a_1f_1+\cdots+a_if_i=0.\]

The main result in this paper can be stated as follows:

\begin{thm}\label{t:main}
Let $\{f_1,\ldots,f_k\}\subset\Mod_{g,p}$ be an irredundant collection of Dehn twists or pseudo-Anosov homeomorphisms supported on connected subsurfaces of a given surface.  Then there exists an $N$ such that for each $n\geq N$, the set of mapping classes $\{f_1^n,\ldots,f_k^n\}$ is a right-angled Artin system for a right-angled Artin subgroup of $\Mod_{g,p}$.  The isomorphism type of this right-angled Artin group is given by $A(\gam)$, where $\gam$ is the coincidence correspondence of $\{f_1,\ldots,f_k\}$.
\end{thm}

We remark that in the hypotheses of Theorem \ref{t:main}, irredundancy is equivalent to the requirement that no two mapping classes $\{f_i,f_j\}$ for $i\neq j$ generate a cyclic subgroup of $\Mod_{g,p}$.  There is a minor issue which arises in relation to two pseudo-Anosov mapping classes $\psi$ and $\phi$ supported on a subsurface $\Sigma'\subset\Sigma$ which commute but which do not share a common power.  In this case, $\psi$ and $\phi$ differ by a multitwist about the boundary components of $\Sigma'$.  We will adopt a convention throughout this paper which disallows such a setup.  If we allow $\psi$ on our list of mapping classes then we will insist that $\psi$ does not twist about the boundary components of $\Sigma'$.  We will consider $\phi$ as a composition of $\psi$ with a multitwist and hence as a mapping class with disconnected support.  However, we will allow the separate twists about the boundary components of $\Sigma'$ to be on our list.  We discuss this issue in more detail in Subsection \ref{ss:twisting}.

Theorem \ref{t:main} says that after perhaps passing to powers, there are few non-obvious relations between mapping classes.  This means that the various relations in mapping class groups such as braid relations, chain relations, lantern relations, etc. are all ``shallow" in the mapping class group.

We will say that a group $G$ is {\bf enveloped by a right-angled Artin group} $A(\gam)$ if $G$ is generated by elements contained in the {\bf clique subgroups} of $\gam$, which is to say the subgroups induced by complete subgraphs of $\gam$.

The following result is an immediate corollary of Theorem \ref{t:main} and the description of reducible mapping classes.

\begin{cor}
Let $\{f_1,\ldots,f_k\}$ be any collection of mapping classes.  Then there exists an $N$ such that for each $n\geq N$, the group $\langle f_1^n,\ldots,f_k^n\rangle$ is enveloped by a right-angled Artin subgroup of the mapping class group.
\end{cor}

As we will show below, it is not true that a group enveloped by a right-angled Artin group is itself a right-angled Artin group, even abstractly.  Furthermore, we will see that any group enveloped by a right-angled Artin group can be embedded in some mapping class group, so the conclusion of Theorem \ref{t:main} is as strong as one could hope.

To get the full conclusion of Theorem \ref{t:main}, it is clear that irredundancy is a necessary hypothesis.  It is clear that the result also fails if any one of the mapping classes $\{f_1,\ldots,f_k\}$ has finite order.  The hypothesis that all the mapping classes be Dehn twists or pseudo-Anosov on subsurfaces is likewise essential, for the reason that not every group enveloped by a right-angled Artin group is itself a right-angled Artin group.

An immediate question raised by Theorem \ref{t:main} asks which right-angled Artin groups occur as subgroups of mapping class groups.  For a fixed genus, ascertaining which right-angled Artin groups which embed in $\Mod_g=\Mod_{g,0}$ appears to be a rather subtle problem.  However, we can show the following:

\begin{prop}\label{p:emb}
Let $A(\gam)$ be a right-angled Artin group.  There exists a $g$ such that $A(\gam)<\Mod_g$, with the image of the vertices of $\gam$ consisting of powers of Dehn twists about simple closed curves or pseudo-Anosov homeomorphisms supported on connected subsurfaces.  It is possible to arrange the embedding in such a way that the image is contained in the Torelli subgroup of $\Mod_g$.
\end{prop}

\subsection{Remarks on embedding right-angled Artin groups in mapping class groups}
In \cite{CP}, it was shown that every right-angled Artin group can be embedded in some right-angled Artin group, with a right-angled Artin system given by a collection of powers of Dehn twists about simple closed curves.
A result of similar ilk can be found in the paper \cite{CW} of Crisp and Wiest.  If $\gam$ is a finite graph, its {\bf complement} (or {\bf dual}) graph $\gam^*$ is given by embedding $\gam$ into the complete graph on the vertices of $\gam$, and then deleting the edges which belong to $\gam$.  Crisp and Wiest prove that if $\gam^*$ embeds in a surface of genus $g$ then the right-angled Artin group $A(\gam)$ embeds in a genus $g$ surface braid group.  The problem of embedding various right-angled Artin groups into mapping class groups is the subject of Crisp and Farb's preprint \cite{CF}.  Sabalka showed in \cite{Sab} that right-angled Artin groups can always be embedded into graph braid groups.

As an immediate corollary to Proposition \ref{p:emb}, we obtain:
\begin{cor}
Let $A(\gam)$ be a right-angled Artin group and let \[\langle g_1,\ldots,g_k\rangle=G<A(\gam)\] be a finitely generated subgroup.  Then there exists an $N$ such that for all $n\geq N$, we have that \[\langle g_1^{n},\ldots, g_k^{n}\rangle\] is enveloped by a right-angled Artin group.
\end{cor}
We record this corollary because it is well--known that there are many extremely complicated subgroups of right-angled Artin groups when the flag complex associated to the right-angled Artin system is not simply connected (see \cite{BeBr}, also \cite{Droms1}).  A systematic study of right-angled Artin subgroups of a given right-angled Artin group has recently been carried out by S. Kim and the author in \cite{KK}.

\subsection{Commensurability questions}

Recall that two groups $G$ and $H$ are {\bf commensurable} if there exist finite index subgroups $G'$ and $H'$ of $G$ and $H$ respectively such that $G'\cong H'$.  Since $\Mod_{g,p}$ is finitely generated by Dehn twists, it is therefore natural to ask whether or not $\Mod_{g,p}$ is commensurable with a right--angled Artin group.  More generally, one can wonder whether $\Mod_{g,p}$ virtually injects into a right--angled Artin group.

C. Leininger has informed the author that combining the results of Crisp and Wiest in \cite{CW} and Leininger and Reid in \cite{LR}, one can show that $\Mod_{2,0}$ is not commensurable with a right-angled Artin group.  A proof of this fact appears in \cite{CLM}.
For higher genera, an even stronger conclusion is true:

\begin{thm}\label{t:notvirt}
If $g>2$ or if $g=2$ and $p>1$ then $\Mod_{g,p}$ does not virtually inject into a right--angled Artin group.
\end{thm}

We thus recover a corollary to a result of M. Kapovich and B. Leeb (see \cite{KL}), namely that a sufficiently complicated ($g>2$ or $g=2$ and $p>1$) mapping class group does not embed into any right-angled Artin group.  Kapovich and Leeb prove a slightly different statement, namely that sufficiently complicated mapping class groups (which is to say $g=2$ and $p\geq 1$ or $g>2$) do not act effectively, cocompactly and discretely on Hadamard spaces, which is to say Riemannian manifolds with everywhere nonpositive sectional curvatures.
In modern terminology, Theorem \ref{t:notvirt} can be stated as ``mapping class groups are not virtually special" (cf. the recent work of Haglund and Wise in \cite{HW}).

Theorem \ref{t:notvirt} may hold for closed surfaces of genus two, but the methods in this paper do not work when $\Sigma$ has genus $2$, and the author is not aware of a way to rectify this difficulty.  Note that we have $\Mod_{1,0}\cong\Mod_{1,1}\cong SL_2(\bZ)$, which is virtually a free group and hence commensurable with a right-angled Artin group.

Along lines similar to Theorem \ref{t:notvirt}, M. Bridson has shown in \cite{Brid2} that whenever the mapping class group of a closed orientable surface of genus $g$ acts by semisimple isometries on a complete $\CAT(0)$ space of dimension less than $g$, it has a global fixed point.  Actions of mapping class groups by isometries on $\CAT(0)$ spaces are also studied in \cite{Brid1}.  Bridson also shows that a genus $2$ mapping class group acts properly by semisimple isometries on an $18$--dimensional complete $\CAT(0)$ space.

\subsection{Motivation from linear groups}
One of the motivations for Theorem \ref{t:main} comes from linear groups and the {\bf Tits alternative} (see \cite{Tits}), which was originally formulated for finitely generated linear groups:

\begin{thm}\label{t:tits}
Let $G$ be a finitely generated linear group.  Then $G$ is either virtually solvable or contains a nonabelian free group.
\end{thm}

It is known that mapping class groups satisfy the Tits alternative.  One of the main results of Birman Lubotzky and McCarthy (see \cite{BLM}) is the following:
\begin{thm}
Let $G<\Mod_{g,p}$ be virtually solvable.  Then $G$ is virtually abelian.  Furthermore, the rank of a free abelian subgroup of $G$ is bounded by the size of a maximal collection of disjoint, non--peripheral, pairwise non--isotopic simple closed curves on $\Sigma_{g,p}$.
\end{thm}

A subgroup of the mapping class group which is virtually abelian but not virtually cyclic cannot contain pseudo-Anosov elements.  By the work of Ivanov (see \cite{I}), such a subgroup must be be reducible, which means that it stabilizes a finite nonempty collection of isotopy classes of essential simple closed curves.  Ivanov also proved that every subgroup of $\Mod_{g,p}$ which is irreducible and not virtually cyclic contains a nonabelian free group, completing the Tits alternative for $\Mod_{g,p}$.

With our setup, the Tits alternative can be refined to read as follows:

\begin{thm}
Let $G<\Mod_{g,p}$ be generated by $\{f_1,\ldots,f_k\}$, where each $f_i$ has infinite order.  Then exactly one of the following holds:
\begin{enumerate}
\item
$G$ is virtually torsion--free abelian.
\item
There is an $N$ such that for all $n\geq N$, we have that $\{f_1^{n},\ldots, f_k^{n}\}$ generates a nonabelian subgroup $H<G$ which is enveloped by a right-angled Artin subgroup of $\Mod_{g,p}$.
\end{enumerate}
\end{thm}

Any two--generated nonabelian group enveloped by a right-angled Artin group is automatically free, as can be deduced from Lemma \ref{l:pingpong}.  Thus, we do indeed obtain a strengthening of the Tits alternative.

\subsection{Decision problems}
Another perspective on the discussion in this paper stems from decision problems in group theory.  In general, mapping class groups are so complicated that there is no algorithm for computing whether or not two finitely generated subgroups are isomorphic.  We shall prove a more general fact:

\begin{prop}\label{p:f2f2}
Let $G$ be a finitely presented group which contains an embedded product of two nonabelian free groups.  There is no algorithm which determines whether two finitely generated subgroups of $G$ are isomorphic.
\end{prop}

We show that solving the isomorphism problem for finitely generated subgroups of $\Mod_{g,p}$ solves the generation problem for $F_6\times F_6$, which is known to be unsolvable by the work of Lyndon and Schupp (\cite{LySch}) and C. Miller (\cite{Mil}).  We will use this latter fact to show that if a group contains a copy of $F_2\times F_2$ then it has an unsolvable isomorphism problem for finitely generated subgroups by showing that the isomorphism problem for finitely generated subgroups of $F_2\times F_2$ solves the generation problem for $F_6\times F_6$.

On the other hand, the isomorphism type of a right-angled Artin group can be algorithmically determined.  To illustrate this fact, we prove the following rigidity result:
\begin{thm}\label{t:rigidhomo}
Let $A(\gam)$ and $A(\Delta)$ be two right-angled Artin groups.  Then $A(\gam)\cong A(\Delta)$ if and only if $H^*(A(\gam),\bQ)\cong H^*(A(\Delta),\bQ)$ as algebras, which happens if and only if $\gam$ and $\Delta$ are isomorphic as graphs.
\end{thm}

Theorem \ref{t:rigidhomo} can be found in Sabalka's paper \cite{Sab1}, and a proof follows from a combination of theorems of Droms in \cite{Droms1} and Gubeladze in \cite{Gub}.

The fact that $A(\gam)\cong A(\Delta)$ if and only if $\gam\cong\Delta$ was proved by C. Droms in \cite{Droms2}.  Combining Theorem \ref{t:rigidhomo} with Droms' result, we see that to specify the isomorphism type of a right-angled Artin group, the rational cohomology algebra and the graph both suffice.

If $\{f_1,\ldots,f_n\}$ are mapping classes which are a right-angled Artin system for a right-angled Artin group $A(\gam)$, the structure of $\gam$ can be gleaned from the intersection theory of subsurfaces and reduction systems for these mapping classes.  It follows that Theorem \ref{t:main} and Theorem \ref{t:rigidhomo} together provide a partial solution to an ``up to powers" isomorphism problem for finitely generated subgroups of the mapping class group, although the isomorphism problem for finitely generated subgroups of the mapping class groups is unsolvable.

\subsection{Analogies with lattices and symmetric spaces}
The final perspective on Theorem \ref{t:main} which we explore comes from the analogy between Teichm\"uller spaces and symmetric spaces, and between mapping class groups and lattices in semisimple Lie groups.  We have seen that each right-angled Artin group is contained in a mapping class group, so we can ask which right-angled Artin groups are contained in lattices in semisimple Lie groups.  Similarly, we can ask whether a statement analogous to Theorem \ref{t:main} holds.  Questions of the former ilk have been discussed by T. Januszkiewicz and J. \'Swiatkowski in \cite{JS} and by S. Wang in \cite{Wang}, for instance.

It turns out that the conclusion of Theorem \ref{t:main} fails dramatically for lattices in $SL_n(\bR)$ for instance, when $n\geq 3$.  The following result is due to H. Oh in \cite{HeeOh}.  The precise definitions of all the terms are not important for understanding the spirit of the result:
\begin{thm}
Let $G$ be an absolutely adjoint simple real Lie group of $\bR$--rank at least two and let $U_1,U_2<G$ be opposite horospherical subgroups.  Suppose that $G$ is not split over $\bR$ and that $U_1$ is not the unipotent radical of a Borel subgroup in a group of type $A_2$.

Let $F_1$ and $F_2$ be lattices in $U_1(\bR)$ and $U_2(\bR)$ respectively and let $\gam$ be the subgroup generated by these two lattices.  If $\gam$ is discrete then there is a $\bQ$--form with respect to which $U_1$ and $U_2$ are defined over $\bQ$ and $F_i$ is commensurable with $U_i(\bZ)$ for each $i$.  Furthermore, $\gam$ is commensurable with $G(\bZ)$.
\end{thm}

The requirement in the previous result that $G$ does not split over $\bR$ is essential.  For instance, one obtains immediate generalizations of Theorem \ref{t:main} for groups like $SL_2(\bZ)\times SL_2(\bZ)$, since we can view $SL_2(\bZ)\times SL_2(\bZ)$ as the mapping class group of a stable curve consisting of two tori glued together at a point.

To contrast Oh's result with Theorem \ref{t:main}, the role of the unipotent subgroups could be played by powers of Dehn twists about pairwise intersecting curves.  Theorem \ref{t:main} implies that powers of those Dehn twists generate a free group, which is far from being the appropriate analogue of a lattice in mapping class group theory.  So, Theorem \ref{t:main} can be thought of as a ``rank one" phenomenon for mapping class groups.

For discrete subgroups of rank one semisimple Lie groups, we do get an analogue of Theorem \ref{t:main}, though the right-angled Artin groups which occur are limited.  We have the following results:

\begin{prop}\label{p:hyp}
Let $\{g_1,\ldots,g_m\}\subset SO(k,1)$ generate a discrete subgroup, and suppose furthermore that these elements are irredundant.  Then there exists an $N$ such that for all $n\geq N$, we have that $\{g_1^n,\ldots,g_m^n \}$ is a right-angled Artin system for a free product of finitely generated free abelian subgroups of $SO(k,1)$.
\end{prop}

For some other rank one Lie groups one gets very close but not verbatim analogues of Theorem \ref{t:main}.  If $\mathcal{N}$ is a class of groups and $G$ is a free product of $\mathcal{N}$--groups, we call a subset $S\subset G$ a right-angled Artin system for $G$ if $\langle S\rangle=G$ and if each $s\in S$ is contained in a free factor for some free product decomposition of $G$.

\begin{prop}\label{p:cxhyp}
Let $\{g_1,\ldots,g_m\}\subset SU(k,1)$ generate a discrete subgroup, and suppose furthermore that these elements are irredundant.  Then there exists an $N$ such that for all $n\geq N$, we have that $\{g_1^n,\ldots,g_m^n\}$ is a right-angled Artin system for a free product of finitely generated torsion--free nilpotent subgroups of $\Isom(\bH_{\bC}^k)$ under the projection \[SU(k,1)\to\Isom(\bH_{\bC}^k),\] where $\bH_{\bC}^k$ denotes complex hyperbolic space.
\end{prop}

By analogy to Theorem \ref{t:notvirt}, we will show:

\begin{thm}\label{t:notcomm}
Let $k\geq 3$ and let $\gam<SO(k,1)$ be a lattice.  Then $\gam$ is not commensurable with a right-angled Artin group.
\end{thm}

If one removes the hypothesis of discreteness in any of the results above about rank one Lie groups, the conclusions are false.  In $SL_2(\bR)$, for instance, there are finitely generated (indiscrete) solvable subgroups such that no powers of the generators generate an abelian group.  B. Farb has conjectured an analogue of Theorem \ref{t:main} for $SL_2(\bR)$ which would account for the aforementioned solvable subgroups, but his conjecture appears to be extremely difficult to establish.

\subsection{Notes and open questions}
Right-angled Artin subgroups of mapping class groups have been studied by various authors, for example by Crisp and Paris in \cite{CP}.  Our methods exploit the dynamics of the action of $\Mod_{g,p}$ on the space of projective measured laminations on $\Sigma$, i.e. $\bP\mathcal{ML}(\Sigma)$.  An account of some aspects of this theory can be found in the paper of McCarthy and Papadopoulos (\cite{MP}).  Questions about right-angled Artin groups in mapping class groups generated by powers of mapping classes were considered by L. Funar in \cite{Fun} and M. L\"onne in \cite{Loen}.

A recent result of Clay, Leininger and Mangahas closely related to Theorem \ref{t:main} analyzes right-angled Artin groups which quasi--isometrically embed into mapping class groups and can be found in \cite{CLM}.  The methods of their paper rely heavily on the complex of curves.

An interesting related question is to extend the theory in this paper to $\Out(F_n)$.  It is unclear whether or not the exact analogue of Theorem \ref{t:main} holds for $\Out(F_n)$, and the proof given here will certainly not work for $\Out(F_n)$.  One can deduce an analogue to Theorem \ref{t:notvirt} for $\Out(F_n)$ when $n\geq 4$ as follows: note that $\Aut(F_n)<\Out(F_{n+1})$.  By the result of Formanek and Procesi in \cite{FP}, $\Aut(F_n)$ is not linear when $n>2$, so that $\Out(F_n)$ is nonlinear for $n>3$.  It follows that $\Out(F_n)$ cannot virtually inject into a right-angled Artin group for $n>3$.  One can ask:
\begin{quest}
Does $\Out(F_3)$ virtually embed in a right-angled Artin group?
\end{quest}

A positive answer to this question would immediately imply that $\Out(F_3)$ is linear.  Furthermore, it would immediately imply all sorts of properties for $\Out(F_3)$, such as virtual residually torsion--free nilpotence and the Tits alternative, and it would furnish $\Out(F_3)$ with a properly discontinuous action by isometries on a CAT(0) space.

Finally, we pose the following question:
\begin{quest}
Let $f_1,\ldots,f_k$ be mapping classes.  Is there a uniform $N$ (which depends only on the underlying surface) such that $\{f_1^N,\ldots,f_k^N\}$ generate a subgroup which is enveloped by a right-angled Artin subgroup of $\Mod_{g,p}$?
\end{quest}

This question appears to be open in general.  When $\{\psi,\phi\}$ is a pair of pseudo-Anosov mapping classes, an affirmative answer was given by Fujiwara in \cite{Fujiwara}.  A related statement has been conjectured by Funar in \cite{Fun}.

\section{Acknowledgements}
The author thanks B. Farb for suggesting the problem, for his help and comments, and for sharing the preprint \cite{CF}.  The author also thanks J. Aramayona, M. Bridson, M. Casals, T. Church, M. Clay, U. Hamenst\"adt, C. Leininger, J. Mangahas, C. McMullen, C.F. Miller, A. Putman, M. Sapir, A. Silberstein and K. Vogtmann for useful conversations and comments.  The author especially thanks M. Kapovich for a careful reading of earlier versions of the manuscript.  The author thanks the Hausdorff Research Institute for Mathematics in Bonn and the Institute for Mathematical Sciences in Singapore, where this research was completed.  The author finally thanks the referee for very helpful comments which greatly improved the exposition and correctness of the article.  The author is partially supported by an NSF Graduate Research Fellowship.

\section{The ping--pong lemma as the fundamental tool}
When a free group acts on a set $X$, one often wants to know whether or not it acts faithfully.  Given any candidates for two free factors, if these two factors ``switch" two subsets of $X$ then under some further technical hypotheses, the action will be faithful.
A more general ping--pong argument holds for right-angled Artin groups and can be found in the preprint \cite{CF} of Crisp and Farb.  We recall a proof here for the reader's convenience.  A much more detailed exposition on ping--pong lemmas and the material in Section \ref{s:hyp} can be found in \cite{KobIMS}.  Let $\gam$ be a simplicial graph with vertex set $\{1,\ldots,n\}$ and edge set $E(\gam)$.  We denote the associated right-angled Artin group by $A(\gam)$.  The group $A(\gam)$ is generated by elements $\{g_1,\ldots,g_n\}$ according to the standard presentation.  We will interchangeably call those generators the {\bf vertex generators} or a {\bf right-angled Artin system}.  The ping--pong argument can be stated as follows:

\begin{lemma}\label{l:pingpong}
Let $X$ be a set on which $A(\gam)$ acts.  Suppose there exist (not necessarily disjoint) subsets $X_1,\ldots, X_n\subset X$ and an element \[x\in X\setminus\bigcup_{i=1}^n X_i\] such that for each nonzero integer $k$, we have:
\begin{enumerate}
\item
$g_i^k(X_j)\subset X_j$ if $(i,j)\in E(\gam)$.
\item
$g_i^k(X_j)\subset X_i$ if $(i,j)\notin E(\gam)$.
\item
$g_i^k(x)\in X_i$ for all $i$.
\end{enumerate}

Then $A(\gam)$ acts faithfully on $X$.
\end{lemma}
\begin{proof}[First proof of Lemma \ref{l:pingpong}]
Under the hypotheses it will follow that the orbit of $x$ is a transitive $A(\gam)$--set.  We write $\ell(g)$ for the word length of $g\in A(\gam)$ with respect to the standard generating set.  It suffices to show that if $g=g_i^{\pm 1}g'$ with $\ell(g)=\ell(g')+1$, then $g(x)\in X_i$, for then no non-identity element of $A(\gam)$ fixes $x$.

We proceed by induction on $\ell(g)$.  We write $g=g_i^kw$, with $\ell(g)=\ell(w)+|k|$, $k\neq 0$ and $|k|$ maximal.  If $w=1$ then $g(x)=g_i^k(x)\in X_i$.  Otherwise, $w=g_jw'$, with $\ell(w)=\ell(w')+1$ and $i\neq j$.  If $(i,j)\notin E(\gam)$, we obtain that $w(x)\in X_j$ and $g(x)\in X_i$.  If $(i,j)\in E(\gam)$ then $g=g_jg_i^kw'$, where $\ell(g)=\ell(g_i^kw')+1$ so that $g_i^kw'(x)\in X_i$.  It then follows that $g(x)\in X_i$.
\end{proof}

We will appeal several times throughout this paper to an action of a right-angled Artin group on a set which satisfies the hypotheses of Lemma \ref{l:pingpong}.  For the sake of completeness and concreteness, we will now exhibit such an action:

\begin{prop}
Let $\gam$ be a graph with $n$ vertices.  There exists a set $X$, subsets $X_1,\ldots,X_n$ of $X$ and a basepoint $x_0\in X$ which satisfy the hypotheses of Lemma \ref{l:pingpong}.
\end{prop}
\begin{proof}
We let $X$ be the set of reduced words in $A(\gam)$ and we set $x_0$ to be the identity.  $A(\gam)$ acts on this set by left multiplication.  We write $g_1,\ldots,g_n$ for the vertex generators or right-angled Artin system for $A(\gam)$.  We let $X_i$ be the set of words of the form $w$ in $X$ which can be written in reduced form as $g_i^s\cdot w'$, where $s\neq 0$ and the length of $w$ is strictly larger than the length of $w'$.  Let $j\neq i$.  If $(i,j)$ is not an edge in $\gam$ and $w\in X_i$ then $g_j^k\cdot w$ is clearly contained in $X_j$ for all nonzero $k$.  If $(i,j)$ is an edge in $\gam$ then $g_j^k\cdot w$ is contained in $X_i$, since we can commute $g_j$ with $g_i$.

In the latter case, it is possible that $g_j^kw$ is not reduced.  However, if any reduction occurs then we must be able to write a reduced expression $w=g_j^tg_i^sw''$ for some nonzero $t$ and some word $w''$ which is strictly shorter than $w'$.  After we perform all possible reductions to the left of the left--most occurrence of $g_i$, we can move the occurrences of $g_i$ to the far left.
\end{proof}

We now give another perspective on the ping--pong lemma for right-angled Artin groups which was pointed out to the author by M. Kapovich and can also be found in \cite{KobIMS}.

\begin{proof}[Second proof of Lemma \ref{l:pingpong}]
Write $w\in A(\gam)$ as a reduced word in the vertices of $\gam$.  We call a word {\bf central} if it is a product of vertex generators which commute with each other (in other words, the word is a product of vertices which sit in a complete subgraph of $\gam$).  We say that $w\in A(\gam)$ is in {\bf central form} if it is written as a product $w=w_n\cdots w_1$ of central words which is maximal, in the sense that the last letter of $w_j$ does not commute with the last letter of $w_{j-1}$ (where we read from left to right).

Now let $x$ be the basepoint and let $w=w_n\cdots w_1$ be a nontrivial element of $A(\gam)$, written in central form.  We claim that $w(x)$ is contained in the $X_i$ corresponding to the last letter of $w_n$.  We proceed by induction on $n$.

If $n=1$ then conditions $1$ and $3$ of the assumptions on $A(\gam)$ and $X$ gives the claim.  Now we consider $w_n(w_{n-1}\cdots w_1(x))$.  By induction, $w_{n-1}\cdots w_1(x)$ is contained in the $X_i$ corresponding to the last letter $v_i$ of $w_{n-1}$.  The last letter $v_j$ of $w_n$ does not commute with $v_i$, so condition $2$ implies that $w_n(X_i)\subset X_j$, so that $w(x)\in X_j$.  In particular, $w(x)\neq x$.
\end{proof}

\section{Analogy with finite volume hyperbolic $k$-manifolds and rank one Lie groups}\label{s:hyp}
In this section we establish parallel results in the theory of hyperbolic $k$--manifolds and lattices in more general rank one Lie groups.  In this section only, we will be using $\gam$ to denote a lattice and $\Delta$ to denote a graph since the notation of $\gam$ for a lattice seems to be more standard.

\subsection{Real hyperbolic spaces}
Let $\gam<PSL_2(\bC)$ be generated by elements $\{g_1,\ldots,g_m\}$, no two of which share a fixed point at infinity.  If each $g_i$ is a hyperbolic isometry of $\bH^3$ then a standard ping--pong argument proves:
\begin{prop}\label{p:hyphyp}
There exists a positive natural number $N$ such that for all $n\geq N$, we have $\langle g_1^n,\ldots, g_m^{n}\rangle$ is a free group.
\end{prop}

We may assume that every element of a discrete subgroup $\gam$ of $SO(k,1)$ or $SU(k,1)$ is either parabolic or hyperbolic, as discreteness allows us to neglect elliptic elements by passing to a finite index subgroup.  Furthermore, discreteness guarantees that if $g\in\gam$ stabilizes a point $x_0$ at infinity then either $g$ is contained in a parabolic subgroup $\gam_0$ which is the stabilizer of $x_0$ in $\gam$, or $g$ is hyperbolic and the stabilizer of $x_0$ is cyclic.  A discrete parabolic subgroup of $SO(k,1)$ is virtually a free abelian group (see any reference on hyperbolic geometry such as the book \cite{BP} of Benedetti and Petronio).

Recall that Proposition \ref{p:hyp} says that Proposition \ref{p:hyphyp} holds for general real hyperbolic spaces, even if we allow parabolic isometries as well as hyperbolic isometries.
\begin{proof}[Proof of Proposition \ref{p:hyp}]
Label the fixed points of generators of $\gam$ as \[A_1,R_1,\ldots, A_m,R_m\] and \[P_1,\ldots,P_{m'},\] where $(A_i,R_i)$ are the pair of fixed points of a hyperbolic isometry and $P_i$ is the fixed point of a parabolic isometry.  After passing to a finite power of the generators of $\gam$ we may assume that for each $i$, the stabilizer of $P_i$ is a free abelian group.  By discreteness, no $A_i$ is equal to any $R_j$, and the parabolic fixed locus is different from the hyperbolic fixed locus.  Choose small neighborhoods $V_i$ of $(A_i,R_i)$ for each $i$, so that the hyperbolic isometries play ping--pong with each other.

For a fixed $i$, conjugate $P_i$ to be $\infty$.  For any norm on $\bR^k$, there is an $N$ such that all other fixed points of generators of $\gam$ have norm less than $N$.  If the stabilizer of $\infty$ is cyclic and generated by $g$ then there is an $M$ such that $g^{\pm M}(V_i)\subset \{x\mid |x|>N\}$.  Thus, $g^M$ plays ping--pong with the hyperbolic isometries.  Conjugating back, it follows that for any parabolic isometry $g_i$ stabilizing $P_i$ and any compact subset $K$ of $\partial\bH^{k}\setminus P_i$, there is an $M(K)$ such that for all $m\geq M$, we have $g^m(K)\cap K=\emptyset$.

Suppose that the stabilizer $S$ of $\infty$ is isomorphic to $\bZ^s$ for $s>1$.  We have that $S$ acts on $\bR^{k-1}=\partial\bH^k\setminus\{\infty\}$ by fixed--point free Euclidean isometries.  By discreteness and residual finiteness of finitely generated abelian groups, we see that for any parabolic subgroup $S$ stabilizing $P_i$ and any compact subset $K\subset\partial\bH^{k}$, there is a finite index subgroup $S_K$ of $S$ such that for any $1\neq g\in S_K$, we have $g(K)\cap K=\emptyset$.

Ping--pong is thus setup by taking the $\{X_i\}$ to be the $\{V_i\}$ for the hyperbolic isometries, and small neighborhoods $\{U_j\}$ of each parabolic fixed point.  We choose these neighborhoods to be small enough so that they are pairwise disjoint and do not cover all of the sphere at infinity.   The complement of of each such neighborhood is a compact subset of $\partial\bH^k$.  It follows that there is an $M$ such that for all $m\geq M$, the isometry $g_i^m$ maps the compact set $\partial\bH^k\setminus X_i$ into $X_i$.  By irredundancy and discreteness, $g_i$ and $g_j$ commute if and only if they generate a parabolic subgroup.  If $g_i$ and $g_j$ do not commute then $X_j\subset\partial\bH^k\setminus X_i$.  Lemma \ref{l:pingpong} gives us the conclusion of the proposition.
\end{proof}

The requirement that the isometries do not share fixed points at infinity is essential.  For instance, the isometries $z\mapsto z+1$ and $z\mapsto 2z$ of $\bH^2$ have many nontrivial relations, as do their powers.  These two isometries generate a solvable group which is not virtually abelian.  More generally, the stabilizer of a point at infinity in the group of isometries of $\bH^2$ is a two--step solvable Lie group (the affine group).  It can be shown that the only discrete subgroups of this affine group are virtually abelian (cf. S. Katok's book \cite{Ka}).

\subsection{Complex hyperbolic spaces}
We now make a short digression into complex hyperbolic geometry so that we can make sense of right-angled Artin systems for discrete subgroups of $SU(k,1)$ and prove Proposition \ref{p:cxhyp}.  For more details, consult the book \cite{Gold} of Goldman, for instance.  Recall that in analogy to the usual hyperbolic space which we call {\bf real hyperbolic space}, there is a notion of {\bf complex hyperbolic space}.  It is often denoted $\bH_{\bC}^k$.  Like real hyperbolic space it can be constructed from a certain level set of the indefinite inner product on $\bC^{k,1}$.  There is again a notion of a boundary $\partial\bH_{\bC}^k$ (which is topologically a sphere) to which the isometry group of $\bH_{\bC}^k$ extends.  Isometries are classified analogously to isometries of real hyperbolic space, into hyperbolic, elliptic and parabolic types.  If $\gam$ is a finitely generated discrete group of isometries, then $\gam$ virtually contains no elliptic elements.

The stabilizer of a point $p$ in $\partial\bH_{\bC}^k$ within the isometry group of complex hyperbolic space is identified with a real $(2k-1)$--dimensional Heisenberg group $H$.  This group is a $2$--step nilpotent real Lie group which admits a real basis \[x_1,y_1,\ldots,x_{k-1},y_{k-1},z,\] satisfying the relations $[x_i,y_i]=z$ for each $i$, and all other commutators being trivial.  The Heisenberg group acts simply transitively on $\partial\bH_{\bC}^k\setminus p$, so that $\partial\bH_{\bC}^k$ can be viewed as the one--point compactification of $H$.  Note that we may assume that the discrete parabolic subgroups of the isometry group of $\bH_{\bC}^k$ are all torsion--free and at most two--step nilpotent.

Let $\gam_0$ be a discrete parabolic group of isometries of $\bH_{\bC}^k$ fixing a point $p\in\partial\bH_{\bC}^k$.  Without loss of generality, $\gam_0$ is a torsion--free lattice in $H=\partial\bH_{\bC}^k\setminus p$.  Let $g_1,\ldots,g_m$ generate $\gam_0$.  Since $\gam_0$ is torsion--free, $H$ is isomorphic to the real Mal'cev completion of $\gam_0$, so that we may assume that $\gam_0$ is the group of integer points $H_{\bZ}$ of $H$.

Note that $H_{\bZ}$ is residually finite.  Discreteness and residual finiteness imply that if $K$ is any compact subset of $H$ and $x\in K$, there is a finite index normal subgroup $N_{x,K}\subset H_{\bZ}$ such that if $1\neq n\in N_{x,K}$, then $n\cdot x\notin K$.  In fact, we may choose one such $N_{x,K}$ which works for any $x\in K$.  Fixing $K$, if we choose sufficiently large powers of given elements of $H_{\bZ}$, we may assume they lie in $N_{x,K}=N_K$.

Before we give the proof of Proposition \ref{p:cxhyp}, we state a version of the ping--pong lemma which applies to general free products, and whose proof can be found in \cite{dlH}:

\begin{prop}\label{p:freeprod}
Let $G$ be a group acting on a set $X$, let $\gam_1$ and $\gam_2$ be two subgroups of $G$, and let $\gam$ be the subgroup generated by $\gam_1$ and $\gam_2$.  Assume that $\gam_1$ contains at least three elements and $\gam_2$ contains at least two elements.  Assume furthermore that there exist two non--empty subsets $X_1$ and $X_2$ of $X$ with $X_2$ not included in $X_1$ such that $\gamma_1(X_2)\subset X_1$ for every nonidentity $\gamma_1\in\gam_1$ and $\gamma_2(X_1)\subset X_2$ for every nonidentity $\gamma_2\in\gam_2$.  Then $\gam\cong\gam_1*\gam_2$.
\end{prop}

We are now ready to prove Proposition \ref{p:cxhyp}.

\begin{proof}[Proof of Proposition \ref{p:cxhyp}]
Let $\gam<SU(k,1)$ be discrete.  Without loss of generality, $\gam$ is torsion--free.  Let $\gam$ be generated by $\{g_1,\ldots,g_m\}$, which we identify with an irredundant collection of isometries of $\bH_{\bC}^k$ via the projection \[SU(k,1)\to\Isom(\bH_{\bC}^k).\]  We group these isometries according to whether they are parabolic or hyperbolic, and then according to which points they stabilize in $\partial\bH_{\bC}^k$.  We claim that for $n$ sufficiently large, we have \[\langle g_1^n,\ldots,g_m^n\rangle\cong F_{\ell}*N_1*\cdots N_s,\] where $\ell$ is the number of distinct hyperbolic geodesics stabilized by elements of $\{g_1,\ldots,g_m\}$, and each $N_i$ is a subgroup of the stabilizer of some parabolic fixed point of certain elements in the collection $\{g_1,\ldots,g_m\}$.

Fixed points of hyperbolic isometries come in pairs, one element of which is attracting and the other of which is repelling.  It follows that we can find neighborhoods of these fixed points and powers of the hyperbolic generators of $\gam$ so that they play ping--pong with each other.

Let $N$ be a discrete parabolic subgroup of $\Isom(\bH_{\bC}^k)$ and let $K\subset\partial\bH_{\bC}^k$ be a compact set which does not contain the fixed point $p$ of $N$.  Let $U$ be a neighborhood $p$ which is disjoint from $K$.  By the residual finiteness of $N$, we can find a finite index subgroup $N'<N$ such that every nonidentity element of $N'$ sends $K$ into $U$.  If $g_i\in N$ is a generator of $\gam$ then there is an $n$ such that $g_i^n\in N'$.

Choose disjoint neighborhoods of all the fixed points of elements in the collection $\{g_1,\ldots,g_m\}$ and an $n$ sufficiently large such that if $U$ and $V$ are neighborhoods of the distinct fixed point(s) of $g$ and $h$ respectively, $g^{\pm n}(V)\subset U$ and $h^{\pm n}(U)\subset V$.  This is possible for hyperbolic and parabolic isometries by the previous two paragraphs.  But then the hypotheses of Lemma \ref{p:freeprod} are satisfied, giving the desired conclusion.
\end{proof}

\subsection{Commensurability and embeddings}
We now return to the world of real hyperbolic manifolds.  Let $\gam$ be the fundamental group of a finite volume hyperbolic $3$-manifold.  If $\bH^3/\gam$ is closed then there are no noncyclic abelian subgroups of $\gam$.  Thus if $\gam$ is virtually isomorphic to a right-angled Artin group then $\gam$ is virtually free.  This violates Mostow rigidity (see \cite{BP}), since free groups have infinite outer automorphism groups.

We will first determine the right-angeld Artin groups which do embed as discrete groups of isometries of real hyperbolic spaces.  In \cite{JS}, Januszkiewicz and \'Swiatkowski study Coxeter groups of high virtual cohomological dimension.  They prove certain dimension restrictions on Coxeter groups which are also virtual Poincar\'e duality groups in both the right-angled and non-right-angled cases.  Along similar lines, we can characterize the right-angled Artin groups which admit a discrete, faithful representation into $SO(k,1)$, which we call {\bf hyperbolic right-angled Artin groups}.

\begin{prop}\label{p:hypgrp}
The hyperbolic right-angled Artin groups are precisely the finite free products of free abelian groups.
\end{prop}
\begin{proof}
By choosing $k$ sufficiently large, it is clear that a finite free product of free abelian groups can be embedded into $SO(k,1)$.  Conversely, suppose $A(\Delta)$ is a hyperbolic right-angled Artin group.  Let $A_N$ be the isomorphism type of the group given by replacing each element of the right-angled Artin system by its $N^{th}$ power.  It is evident that there is an isomorphism $A(\Delta)\cong A_N$, and this can be verified by a ping--pong calculation.  If the elements of the right-angled Artin system for $A(\Delta)$ do not already play ping--pong at infinity as in the proof of Proposition \ref{p:hyp}, we may replace them by a sufficiently large power so that they do.  It follows that $A(\Delta)$ has to be a finite free product of free abelian groups.
\end{proof}

I. Agol and D. Wise have recently announced a proof that hyperbolic $3$-manifolds are virtually special and hence virtually inject into right-angled Artin groups.  However:

\begin{prop}\label{p:threenotcomm}
Let $\gam<\Isom(\bH^3)$ be a lattice.  Then $\gam$ is not commensurable with a right-angled Artin group.
\end{prop}
\begin{proof}
After passing to a finite index subgroup, we may assume there is no torsion in $\gam$.  Consider $1\neq g\in\gam$.  Since the elements of $\gam$ which centralize $g$ must have the same fixed points on $\widehat{\bC}$ as $g$, we have that the centralizer $C_{\gam}(g)\cong\bZ$ when $g$ is hyperbolic or $C_{\gam}(g)\cong\bZ^2$ when $g$ is parabolic.  If $\gam$ is commensurable to a right-angled Artin group $A(\Delta)$ then $A(\Delta)$ cannot have abelian subgroups of rank greater than two.  In particular $\Delta$ cannot contain a complete graph on three vertices.

If $h,h'$ both centralize with $g$ then both $h$ and $h'$ have the same fixed points as $g$ and thus commute with each other.  It follows that $C_{\gam}(g)$ is abelian.  In particular, $\Delta$ cannot have a non--backtracking path of length more than two.  It follows that $A(\Delta)$ splits as a free product of finitely many copies of $\bZ^2$.  By the work of Droms in \cite{Droms1} (or by the Kurosh Subgroup Theorem), every finitely generated subgroup of $A(\Delta)$ is a right-angled Artin group, so that $\gam$ is virtually a right-angled Artin group.  The previous considerations show that $\gam$ must split as a nontrivial free product of cyclic groups and copies of $\bZ^2$.  Any such free product has an infinite outer automorphism group, which again violates Mostow Rigidity.
\end{proof}

We would like to apply Mostow Rigidity to establish Proposition \ref{p:threenotcomm} for a more general class of hyperbolic manifolds.  Before giving the proof of Theorem \ref{t:notcomm}, we first make a few remarks.  Let $X$ be a right-angled Artin system for a right-angled Artin group.  Let $\al$ be a {\bf Whitehead automorphism} of a right-angled Artin group, which is to say one which restricts to $X\cup X^{-1}$ as one of the following:
\begin{enumerate}
\item
A permutation of $X\cup X^{-1}$.
\item
There is an $a\in X$ such that for each $x\in X$, \[\al(x)\in \{x, ax, xa^{-1}, axa^{-1}\}.\]
\end{enumerate}
It is known by the work of M. Laurence in \cite{Lau} that Whitehead automorphisms generate the whole automorphism group of the right-angled Artin group (see also the work of M. Day in \cite{Day}).

Let $\gam$ be a lattice in $SO(k,1)$ and let $A(\Delta)$ be a candidate right-angled Artin group which is commensurable with $\gam$.  We may assume that there is a finite index subgroup $G$ of $A(\Delta)$ such that $\gam\cong G$, since a finite index subgroup of a lattice is still a lattice.  If $n\geq 3$ then  Mostow Rigidity implies that $\Out(\gam)$ is finite.  Whenever there exists an infinite order, non--inner Whitehead automorphism, it can be used to show that $\Out(G)$ is in fact infinite, which gives us a contradiction.

One example of a Whitehead automorphism of a right-angled Artin group $A(\Delta)$ is a {\bf partial conjugation}.  Let $x$ be a vertex of $\Delta$.  One conjugates the generators of $A(\Delta)$ corresponding to the vertices in a (union of) connected component of the complement of the star of $x$ by the generator corresponding to $x$.  Whenever $x$ separates $\Delta$, this gives a non--inner automorphism of $A(\Delta)$ with infinite order in $\Out(A(\Delta))$.

Another example of a Whitehead automorphism is a {\bf dominated transvection}.  We say a vertex $w$ dominates a vertex $v$ whenever the link of $v$ is contained in the star of $w$.  When $w$ dominates $v$, the map $v\mapsto vw$ is an infinite order non--inner automorphism of the right-angled Artin group, and is called a dominated transvection.  It is clear that this transvection acts nontrivially on the homology of the right-angled Artin group.  It follows that a dominated transvection has infinite order in $\Out(A(\Delta))$.

Usually, the outer automorphism group of a right-angled Artin group admits no non--inner partial conjugacies and no dominated transvections.  As a consequence, one can show that usually the outer automorphism group of a right-angled Artin group is finite, as is shown in the recent paper of R. Charney and M. Farber \cite{ChFar}.  However, the class of hyperbolic right-angled Artin groups is very restricted, and we will see that hyperbolic right-angled Artin groups are far from generic.

\begin{proof}[Proof of Theorem \ref{t:notcomm}]
Suppose that a lattice $\gam$ in $SO(k,1)$ is commensurable with a right-angled Artin group $A(\Delta)$.  Then $A(\Delta)$ must virtually embed as a discrete subgroup of $SO(k,1)$, in which case it is a free product of free abelian groups.  By the Kurosh Subgroup Theorem, it follows that $\gam$ splits as a free product of free abelian groups.  Any such group admits an infinite group of outer automorphisms (coming from dominated transvections, for instance), a violation of Mostow rigidity.
\end{proof}

Note that Theorem \ref{t:notcomm} is absolutely false for $k=2$.  Indeed, $PSL_2(\bZ)$ is virtually a right-angled Artin group and is also a lattice in $PSL_2(\bR)$.  The reason for this dramatic failure is the fact that maximal parabolic subgroups and stabilizers of hyperbolic geodesics in $PSL_2(\bZ)$ are indistinguishable from the point of view of their abstract isomorphism type.
Since Mostow rigidity works in dimension three and higher, it is evident that the proof above does not work for $k=2$.
We can give an alternative proof of Theorem \ref{t:notcomm} which does not use Mostow rigidity so heavily.

\begin{proof}[Second proof of Theorem \ref{t:notcomm}]
Note that we may assume that $\bH^k/\gam$ is not closed.
If $\gam$ is a lattice in $SO(k,1)$, we may assume that the maximal rank of a free abelian group in $\gam$ is $k-1$.  Furthermore, any abelian subgroup of $\gam$ of rank $k-1$ is conjugate into a peripheral subgroup (this is where the argument fails for $k=2$).  It follows that any candidate right-angled Artin group $A(\Delta)$ has cohomological dimension at most $k-1$, which is to say that its Salvetti complex is built out of tori of dimension no larger than $k-1$.  We are assuming that $\gam$ is not cocompact, so we may double it over the cusps.  The resulting group $D(\gam)$ is the fundamental group of a closed, orientable $k$--manifold.  Therefore, $H^k(D(\gam),\bZ)\cong\bZ$.  On the other hand, we can construct a complex with fundamental group $D(\gam)$ out of two copies of a finite cover of the Salvetti complex of $A(\Delta)$.  The resulting complex has cohomological dimension at most $k-1$, a contradiction.
\end{proof}

Given the discussion in this section, Theorem \ref{t:main} should be thought of as a rank one phenomenon of the mapping class group. It is well-known that mapping class groups share many properties with lattices in rank one semisimple Lie groups as well as properties of lattices in higher rank semisimple Lie groups (see the paper \cite{FLM} of Farb, Lubotzky and Minsky, for instance).

\section{The bad behavior of right-angled Artin groups}
The purpose of this section is to give examples of badly--behaved right-angled Artin groups, to show that Theorem \ref{t:main} is as general as possible, and to record some other useful facts about right-angled Artin groups.  A much more detailed discussion of some of the issues arising here can be found in \cite{KK}.

In our analysis of mapping classes and their interactions, it will be helpful to write mapping classes as compositions of commuting mapping classes, so that we only need to consider the ``simplest" types of mapping classes.  In this section, we will consider certain natural subgroups of right-angled Artin groups generated by clique subgroups.  We will be able to give sufficient (though not necessary: see \cite{KK}) conditions for such a subgroup to again be a right-angled Artin group (see Lemma \ref{l:raag} below).  After this section, we will restrict our considerations to the case where each mapping class is a Dehn twist about a simple closed curve or a pseudo-Anosov homeomorphism when restricted to a connected subsurface.

If $\gam$ is a graph, we call a nonempty complete subgraph of $\gam$ a {\bf clique} of $\gam$.  The subgroup generated by those vertices is called a {\bf clique subgroup}.  A {\bf join} of two graphs is a union of two graphs $\gam_1$ and $\gam_2$, together with edges connecting every vertex of $\gam_1$ with every vertex of $\gam_2$.  A {\bf join subgroup} of $A(\gam)$ is a subgroup generated by the vertices of a subgraph of $\gam$ which decomposes as a nontrivial join, considered up to conjugacy.  

In a right-angled Artin group, commutation of elements is more subtle than containment in a conjugate of a clique subgroup.  Behrstock and Charney prove the following result, which is very useful in understanding commutation in right-angled Artin groups:

\begin{prop}[\cite{BehrChar}]\label{p:bc}
Let $g\in A(\gam)$ be cyclically reduced.  The following are equivalent:
\begin{enumerate}
\item
$g\in A(\gam)$ is contained in a join subgroup.
\item
The centralizer $C(g)$ is contained in a join subgroup.
\item
The centralizer of $g$ is not cyclic.
\end{enumerate}
\end{prop}

Let $\gam$ be a finite graph and let $A(\gam)$ be the corresponding right-angled Artin group.  Identify the vertices of $\gam$ with a right-angled Artin system for $A(\gam)$ and hence a set of generators for $A(\gam)$.
Let $C_1,\ldots,C_n$ be the cliques of $\gam$, and let $\{z_{i,j}\}$, $1\leq i\leq n$ and $1\leq j\leq m(i)$ be (possibly empty) collections of (formal positive powers of) vertices (or their formal inverses) in $\{C_i\}$.  Let \[\{Z_i=\prod_j z_{i,j}\},\] and form a graph $\Lambda$ from the collection $\{Z_i\}$ by taking one vertex for each generator and connecting two vertices if the two elements commute in $A(\gam)$.  We will say that a collection $\{Z_i\}$ has {\bf property PP} (which stands for ping--pong) if $\Lambda$ is occurs as an induced subgraph of $\gam$ in a way such that the vertex $v_i$ of $\Lambda$ corresponding to $Z_i$ lies in the support of $Z_i$ under the inclusion $\Lambda\to\gam$.  Here, a subgraph $\Lambda\subset\gam$ is {\bf induced} if two vertices in $\Lambda$ are adjacent if and only if they are adjacent in $\gam$.

\begin{lemma}\label{l:raag}
Let $\{Z_i\}$ have property PP.  Then \[\langle\{Z_i\}\rangle\cong A(\Lambda).\]
\end{lemma}
\begin{proof}
There is clearly a surjective map $A(\Lambda)\to \langle\{Z_i\}\rangle$, which we claim is also injective.  Make $A(\gam)$ act faithfully on a set $X$, and for each vertex $z$ of $\gam$, let $X_z\subset X$ be the subset which witnesses the ping--pong lemma for the action.  Since $\Lambda$ is an induced subgraph of $\gam$, let $z_i$ be the image of the vertex $v_i$ under the inclusion $\Lambda\to\gam$.  Write each $Z_i$ and its inverse in such a way so that all of the vertex letters $z_i$ corresponding to $v_i$ are as far to the right as possible.  Let $x_0$ be the basepoint of the ping--pong action.  Clearly for all nonzero powers of $Z_i$, the image of $x_0$ is in $X_{v_i}$.

If $v_i$ and $v_j$ are connected in $\Lambda$ then the vertices in the supports of $Z_i$ and $Z_j$ form one large clique (this fact is not entirely obvious but can be deduced from Proposition \ref{p:bc}.  See \cite{KK} for more details).  Therefore, each nonzero power of $Z_j$ preserves $X_{v_i}$.

If $v_i$ and $v_j$ are not connected in $\Lambda$ then the vertices $z_i$ and $z_j$ are not connected in $\gam$.  It follows that each nonzero power of $Z_j$ sends $X_{v_i}$ into $X_{v_j}$.  This verifies the ping--pong lemma and proves the result.
\end{proof}

We now illustrate some examples of subgroups of right-angled Artin groups which one might expect to be right-angled Artin groups themselves but which are not, or which are right-angled Artin groups of an unexpected isomorphism type.

Let $A,B,C$ generate a copy of $\bZ^3$ and let $B,C,X,Y$ generate a copy of $\bZ^4$.  Let $A(\gam)$ be the corresponding right-angled Artin group generated by $A,B,C,X,Y$, which is to say there are no relations between $A$ and $\langle X,Y\rangle$.  Consider the subgroup $G$ generated by $AB,AC,X,Y$.  One checks easily that $AB$ and $AC$ both do not commute with $X$ and $Y$.  The na\"ive guess as to the isomorphism type of $G$ would by $\bZ^2*\bZ^2$.  This is not right, however.  Indeed, $(AC)^{-1}AB=BC^{-1}$, which commutes with both $X$ and $Y$.  It follows that $X$ and $Y$ do not generate a maximal abelian subgroup, but rather $X,Y,BC^{-1}$ do.  It follows that if $G$ is a right-angled Artin group, it must be the one whose underlying graph $\gam'$ is a triangle with a single edge emanating out of one of the vertices.  The vertices should be labeled $X,Y,BC^{-1}$ on the triangle, and $BC^{-1}$ should be connected to $AC$.

\begin{prop}
There is an isomorphism $G\cong A(\gam')$.
\end{prop}
\begin{proof}
It is crucial that we already know that $A(\gam)$ is a right-angled Artin group.  It is easy to check that the vertex $C$ dominates both $A$ and $B$.  It follows that the dominated transvections $A\mapsto AC^{\pm1}$ and $B\mapsto BC^{\pm1}$ give automorphisms of $A(\gam)$.  It follows that the group generated by $\{X,Y,BC^{-1},AC\}$ is isomorphic to the subgroup generated by $\{X,Y,B,A\}$, which is precisely $A(\gam')$.
\end{proof}

We will now give an example illustrating the fact that given an arbitrary collection of mapping classes, there might be no powers which generate a right-angled Artin group.  Let $c_1,c_2,c_3,c_4$ be simple closed curves on a surface $\Sigma$, with \[i(c_1,c_3),i(c_2,c_4)\neq 0,\] $c_2$ disjoint from both $c_1$ and $c_3$, $c_1$ disjoint from $c_4$, and $c_3$ disjoint from $c_4$.  Let $T_1$, $T_2$, $T_3$ and $T_4$ be Dehn twists about these curves.  The proof of Theorem \ref{t:main} will show that there are powers $w$, $x$, $y$ and $z$ of $T_1$, $T_2$, $T_3$ and $T_4$ respectively which generate a right-angled Artin group, and the corresponding graph will be a square, so that $\langle w,x,y,z\rangle\cong F_2\times F_2$.  When it makes sense, we will write the composition of these mapping classes additively.

\begin{prop}
The eight mapping classes $w-x$, $w+x$, $x-y$, $x+y$, $y-z$, $y+z$, $w-z$ and $w+z$ have the property that no nonzero powers generate a right-angled Artin group, in the sense that they are not a right-angled Artin system for any right-angled Artin group.
\end{prop}
\begin{proof}
Let $A_1$ be the group generated by $w-x$ and $w+x$, let $A_2$ be the group generated by $x-y$ and $x+y$, let $A_3$ be the group generated by $y-z$ and $y+z$, and let $A_4$ be the group generated by $w-z$ and $w+z$.  These are all isomorphic to $\bZ^2$.  If powers of these mapping classes generated a right-angled Artin group, it would be one on a finite list of possible isomorphism types.  We can immediately rule out $\bZ^4$, any group which contains a copy of $\bZ^3$, and any right-angled Artin group on a graph with fewer than four edges.  The only possible underlying graph is the square.

Consider the intersection of $A_1$ and $A_2$.  It is easy to see that the intersection is generated by $2x$.  Similarly, the intersection of $A_2$ and $A_3$ is generated by $2y$.  Note that $2x$ and $2y$ together do not generate $A_2$, since \[ [\langle x,y\rangle:A_2]=2\] and \[ [\langle x,y\rangle:\langle 2x,2y\rangle]=4.\]  Suppose we replace $x-y$ and $x+y$ by $n(x-y)$ and $m(x+y)$, and we find the smallest positive integers $a$ and $b$ such that $an=bm$.  Then $\langle n(x-y),m(x+y)\rangle$ contains $(an+bm)x$ and $(an-bm)y$.

Note that $\langle x-y,x+y\rangle$ only contains even multiples of $x$ and $y$.  Therefore, if we want $x+y$ and $x-y$ to be contained in the group generated by the vertex generators, we may assume that both $m$ and $n$ are even but not divisible by four.  But then $(an+bm)$ and $(an-bm)$ are both divisible by four, since $\langle x-y,x+y\rangle$ only contains even multiples of $x$ and $y$.  Alternatively, since $n=2n'$ and $m=2m'$ with both $n'$ and $m'$ odd, we can check case by case that if $an'=bm'$ then $an'+bm'$ and $an'-bm'$ are both even.  Thus, $A_{n,m}=\langle n(x-y),m(x+y)\rangle$ always properly contains \[\langle A_{n,m}\cap\langle x\rangle, A_{n,m}\cap\langle y\rangle\rangle.\]

We can repeat this argument for $A_1$, $A_3$ and $A_4$.  Let $n_1,\ldots,n_4$ and $m_1,\ldots,m_4$ be the multiples by which we replace $w-x,x-y,y-z,w-z$ and $w+x,x+y,y+z,w+z$ respectively.  We have that $\langle n_1(w-x),m_1(w+x)\rangle$ properly contains the subgroup generated by its intersection with the groups $\langle w\rangle$ and $\langle x\rangle$.  It is possible that \[\langle n_2(x-y),m_2(x+y)\rangle\] and \[\langle n_4(w-z),m_4(w+z)\rangle\] correct this deficiency by containing smaller multiples of $w$ and $x$, but in that case at least one of \[\langle n_2(x-y),m_2(x+y)\rangle\] or \[\langle n_4(w-z),m_4(w+z)\rangle\] properly contains the group generated by its intersection with the corresponding vertex groups.  It may be possible that \[\langle n_3(y-z),m_3(y+z)\rangle\] corrects the deficiency, but then \[\langle n_3(y-z),m_3(y+z)\rangle\] properly contains the group generated by its intersection with $\langle y\rangle$ and $\langle z\rangle$.
\end{proof}

The last part of this section serves as a motivation for the definition of property PP and an illustration of some of the bad behavior of right-angled Artin groups.

The author is indebted to M. Casals for the following example.  Consider the group $F_2\times F_2$ generated by $w,x,y,z$ as above.  The elements $w,x,yz$ are not a right-angled Artin system for any right-angled Artin group, and neither are any of their powers.  Denoting $yz$ by $g$, it is easy to check that $w$ commutes with the conjugate of $x$ by $g$, and the same relation holds for all powers of $w$, $x$, and $g$.

\begin{prop}
Consider the group $G_N$ generated by $w^N$, $x^N$ and $g^N$ as above, for $N\neq 0$.  Then $G_N$ is not even abstractly isomorphic to a right-angled Artin group.
\end{prop}
\begin{proof}
We will prove the proposition for $G=G_1$.  The proof in general is exactly the same.  We claim that $G$ is not isomorphic to any right-angled Artin group, for which it suffices to check that $G$ is not $\bZ^2*\bZ$ or $F_2\times\bZ$ since $G$ is obviously neither free nor abelian.  For the second case, it is trivial to understand the centralizer of any element and to show that $G$ has no center.

Suppose that $G$ splits as a free product of $\bZ^2$ and $\bZ$.  The classical Kurosh Subgroup Theorem implies that any copy of $\bZ^2$ inside of $G$ is conjugate to a subgroup of the $\bZ^2$ factor of the splitting.  Alternatively, one can see this by constructing the Cayley graph of $\bZ^2*\bZ$ with the standard generating set and noticing that any copy of the $\bZ^2$ Cayley graph must come from a conjugate of the copy based at the identity.  In $G$, $\langle w,x\rangle$ and $\langle w,x^g\rangle$ are both copies of $\bZ^2$, and it is easy to check that these two copies of $\bZ^2$ are not conjugate in $G$.
\end{proof}

\section{Right-angled Artin groups as subgroups of mapping class groups}
The goal of this section is to prove Theorem \ref{t:main}.  To start out, we recall some background information and make some preliminary simplifying remarks.  A reader familiar with the basics of mapping class groups and the action of $\Mod_{g,p}$ on $\bP\mathcal{ML}(\Sigma)$ may wish to skip to subsection \ref{ss:angle}.

\subsection{Reducible mapping classes}
In the course of proving our main theorem, we may certainly replace $\Mod_{g,p}$ with any finite index subgroup.  In particular, we may assume that all the mapping classes we consider are {\bf pure}.
Recall the description of pure mapping classes:
\begin{lemma}\label{l:class}
Let $f$ be a pure mapping class.  Then there exists a system of disjoint, essential, pairwise non-isotopic curves \[\mathcal{C}=\bigcup\gamma_i\] such that:
\begin{enumerate}
\item
$\mathcal{C}$ is preserved component--wise by $f$.
\item
$\Sigma\setminus\mathcal{C}$ consists of a union of simpler surfaces, each of which is preserved component--wise by $f$.
\item
The restriction of $f$ to any component of $\Sigma\setminus\mathcal{C}$ is either pseudo-Anosov or trivial.
\end{enumerate}
\end{lemma}

The description of pure mapping classes in Lemma \ref{l:class} allows us to characterize (virtual) centralizers in $\Mod_{g,p}$.  Virtually solvable subgroups of $\Mod_{g,p}$ where classified in \cite{BLM}, where it was shown that they are virtually abelian.  We will be requiring an in depth understanding of virtually abelian subgroups of mapping class groups here.  To that end, we need to define a {\bf canonical reduction system} for a mapping class.  For each pure mapping class $f$, we consider the collection $\mathcal{C}$ of all simple closed curves whose isotopy class is preserved by $f$.  Let $\{\mathcal{C}_i\}_{i\in I}$ be the set of maximal (though possibly empty) collections of disjoint simple closed curves which witness the fact that $f$ is a pure mapping class.  The canonical reduction system $\mathcal{C}_R$ is defined by \[\mathcal{C}_R=\bigcap_{i\in I}\mathcal{C}_i.\]  We make a few elementary observations.  Let $\Sigma_j$ be a component of $\Sigma\setminus\mathcal{C}_i$.  If the restriction $f_j$ of $f$ to $\Sigma_j$ is pseudo-Anosov relative to its boundary components, then $\Sigma_j$ cannot be reduced any further.  Notice also that if $\mathcal{C}_k$ is any other reduction system for $f$, then the boundary components of $\Sigma_j$ are contained in $\mathcal{C}_k$.  Indeed, otherwise there would be a component of $\Sigma\setminus\mathcal{C}_k$ which properly contains $\Sigma_j$ for which the inclusion is not a homotopy equivalence, so that $\mathcal{C}_k$ is not a reduction system.  Furthermore, let $c\subset\mathcal{C}_i$ be a simple closed curve.  We have that $c$ contributes two boundary components to one or two components of $\Sigma\setminus\mathcal{C}_i$.  Suppose that the restriction of $f$ to these components is the identity and that the restriction to their union together with $c$ is not a nontrivial power of a Dehn twist about $c$.  Then $\mathcal{C}_i\setminus\{c\}$ is also a reduction system for $f$, so that $c\notin\mathcal{C}_R$.

The following can be found in Ivanov's book \cite{I}, for instance.  We recall the proof here for the reader's convenience.

\begin{lemma}\label{l:cent}
Let $f$ and $g$ be two nontrivial pure mapping classes and let $\mathcal{C}_1$ and $\mathcal{C}_2$ be (possibly empty) canonical reduction systems for $f$ and $g$ with restrictions $\{f_i\}$ and $\{g_i\}$ to the interiors of the components of $\Sigma\setminus\mathcal{C}_i$.  Then $f$ and $g$ virtually commute if and only if one of the following conditions holds:
\begin{enumerate}
\item
$f$ and $g$ are both pseudo-Anosov and share a common nonzero power.  This happens if and only if $f$ and $g$ stabilize the same geodesic in Teichm\"uller space.
\item
$\mathcal{C}_1=\mathcal{C}_2$ and the restrictions $f_i$ and $g_i$ for each $i$ are virtually commuting pseudo-Anosov mapping classes, or one of the two restrictions is trivial.  Note that this case encompasses the first case when the reduction systems are empty.
\item
$\mathcal{C}_1$ is properly contained in $\mathcal{C}_2$ and if $\Sigma_i$ is a component of $\Sigma\setminus\mathcal{C}_1$ which contains at least one component of $\mathcal{C}_2$, then the restriction of $f$ to $\Sigma_i$ is trivial.
\item
$\mathcal{C}_1\cup\mathcal{C}_2$ form part of a pants decomposition of $\Sigma$, and if $c\in\mathcal{C}_1$ is contained in a component $\Sigma_i\subset\Sigma\setminus\mathcal{C}_2$, the restriction $g_i$ of $g$ to $\Sigma_i$ preserves $c$, and conversely after switching the roles of $f$ and $g$.
\end{enumerate}
\end{lemma}
\begin{proof}
Clearly the ``if" direction holds.

If both $f$ and $g$ are pseudo-Anosov and do not stabilize the same Teichm\"uller geodesic, then by replacing $f$ and $g$ by positive powers we obtain a free group (a more general statement is proved in Lemma \ref{l:free}).  

Suppose $\mathcal{C}_1\subset\mathcal{C}_2$ and $\Sigma_i\subset\Sigma\setminus\mathcal{C}_1$ contains at least one component of $\mathcal{C}_2$, and suppose that the restriction of $f$ to $\Sigma_i$ is pseudo-Anosov.  There is a collection of simple closed curves $\mathcal{C}'\subset\Sigma_i$ which is preserved by $g$ but not by $f$.  If $c\in\mathcal{C}'$ then some $N$ we have that the geometric intersection number of $f^n(c)$ and $c$ is nonzero for $n>N$, since $f^n(c)$ converges to the stable lamination of $f$.  In particular, a Dehn twist about $c$ will not commute with $f^n$.  More generally, applying $f$ to $\mathcal{C}'$ does not preserve it, so conjugating $g$ by $f$ will not give back $g$.

Finally, suppose that there are no inclusion relations between $\mathcal{C}_1$ and $\mathcal{C}_2$.  Then either there exist a pair $c_i\in\mathcal{C}_i$ with positive geometric intersection or not.  If not, each $c\in\mathcal{C}_1$ sits non-peripherally in $\Sigma\setminus\mathcal{C}_2$, and similarly after switching the indices.  At least one of the restrictions in $\{f_i\}$ or $\{g_i\}$ is nontrivial, or $f$ and $g$ are powers of Dehn twists about curves in $\mathcal{C}_1$ and $\mathcal{C}_2$.  We may assume that $c\in\mathcal{C}_1$ sits in $\Sigma_i\subset\Sigma\setminus\mathcal{C}_2$ and that the restriction $g_i$ is pseudo-Anosov.  It follows that $f$ and $g$ cannot commute.

Suppose that there is an intersecting pair of simple closed curves in the two reduction systems.  We have that in $\Sigma\setminus \mathcal{C}_1$, $c_1$ contributes two boundary components, and the restriction of $f$ to the interior of the components of $\Sigma\setminus \mathcal{C}_1$ which contains the components of $c_1$ is either pseudo-Anosov or trivial.  If it is trivial, then by our conventions on canonical reducing systems we must have that $f$ does a nonzero power of a Dehn twist about $c_1$.  Since $c_2$ has positive intersection number with $c_1$, we have that $g$ conjugated by $f$ cannot preserve $c_2$.  Therefore we may assume that the restriction is pseudo-Anosov on at least one of the components.  Call such a component $X$.  We have that $c_2$ intersects $X$ is a union of arcs $\mathcal{A}$.  Since $X$ itself is a surface of hyperbolic type and the restriction of $f$ to $X$ is pseudo-Anosov, it follows that $f(\mathcal{A})\neq\mathcal{A}$, even if modify $f$ by arbitrary Dehn twists along curves in $\mathcal{C}_1$.  It follows that $f(c_2)\neq c_2$, so that $f$ and $g$ do not commute.
\end{proof}

We are in a position to establish the following well--known special case of Theorem \ref{t:main}:
\begin{lemma}\label{l:free}
Let $\{f_1,\ldots,f_k\}\in\Mod_{g,p}$ be pseudo-Anosov mapping classes.  Then there exists an $N$ such that for all $n\geq N$, the subgroup of $\Mod_{g,p}$ generated by $\{f_1^{n},\ldots, f_k^{n}\}$ is a free group of rank $r\leq k$.  Equality holds if and only if no pair $\{f_i,f_j\}$ generates a virtually cyclic group for $i\neq j$.
\end{lemma}
\begin{proof}
This is a consequence of the ping--pong argument on $\bP\mathcal{ML}(\Sigma)$, and is fleshed out in more detail in \cite{I}.  If $f_i^n\neq f_j^m$ for all $n,m\neq 0$, then they stabilize distinct pairs of measured laminations in $\bP\mathcal{ML}(\Sigma)$.  The induced dynamics of each pseudo-Anosov map on $\bP\mathcal{ML}(\Sigma)$ is to attract to the stable lamination and repel from the unstable lamination.  Outside of a small neighborhood of both the stable and unstable laminations, the attraction/repulsion is uniform, by compactness.

Let $V_i^u$ and $V_i^s$ be small neighborhoods about the unstable and stable laminations of $f_i$ in $\bP\mathcal{ML}(\Sigma)$.  For a sufficiently large $n_i$, $f_i$ will map all of $\bP\mathcal{ML}(\Sigma)\setminus V_i^u$ into $V_i^s$.  The ping-pong argument now applies to show that $\{f_i^{n_i}\}$ generate a free group.  If $f_i^n=f_j^m$ for some $n$ and $m$, we replace $f_j$ by $f_j^m$ and delete $f_i$ from our list of mapping classes.  We can take $N=\max n_i$.
\end{proof}

At this point it is important to remark about a subtle distinction between pseudo-Anosov mapping classes and hyperbolic isometries of $\bH^k$.  On the one hand, by the results of the expos\'es of Fathi, Laudenbach and Po\'enaru in \cite{FLP}, a single pseudo-Anosov lamination determines the pseudo-Anosov homeomorphism, up to a power.  On the other hand, a fixed point in the sphere at infinity $S_{\infty}$ does not uniquely determine the axis of a hyperbolic isometry of $\bH^k$, and we have noted before that ping--pong fails for pairs of hyperbolic isometries which share exactly one fixed point at infinity.  We never encounter this problem with pseudo-Anosov homeomorphisms.

Let $f$ be a pure mapping class with canonical reduction system $\mathcal{C}$.  We say that $f$ has a {\bf pseudo-Anosov component} if the restriction of $f$ to a component of $\Sigma\setminus\mathcal{C}$ is pseudo-Anosov.
Let $f$ and $g$ be pure mapping classes with pseudo-Anosov components and canonical reduction systems $\mathcal{C}_1$ and $\mathcal{C}_2$.  We say that the two pseudo-Anosov components have {\bf essential overlap} if whenever $\mathcal{C}_1$ and $\mathcal{C}_2$ are arranged to have minimal self--intersection (for instance representing the elements of $\mC_1$ and $\mC_2$ by geodesics in some hyperbolic metric), there are pseudo-Anosov components of $f$ and $g$ which intersect nontrivially.  We say that the essential overlap is {\bf distinct} if when the two pseudo-Anosov components coincide, the two restrictions do not share common power.

To prove Theorem \ref{t:main}, it will be necessary to study the dynamics of mapping classes acting on projective measured laminations and the interactions between the limiting laminations of the mapping classes in question.

\subsection{Convergence in $\bP\mathcal{ML}(\Sigma)$}\label{s:conv}

The following facts are both well--known and useful for our purposes, but we will not be providing proofs:
\begin{prop}
Let $\eps>0$, and let $T$ be a Dehn twist about a curve $c$, which we view as an element of $\bP\mathcal{ML}(\Sigma)$.  Let $K\subset\bP\mathcal{ML}(\Sigma)\setminus c$ be a compact set contained in the open set consisting of laminations $\mL$ with $i(c,\mL)>\eps$.  Then for any neighborhood $U$ of $c$ which is disjoint from $K$ and any sufficiently large $N$ (which depends on $K$ and $U$), we have $T^N(K)\subset U$.
\end{prop}

\begin{prop}
Let $\eps>0$, and let $\psi$ be a pseudo-Anosov supported on a subsurface of $\Sigma$ with stable and unstable laminations $\mL^+$ and $\mL^-$ respectively.  Let $K\subset\bP\mathcal{ML}(\Sigma)\setminus \mL^-$ be a compact set contained in the open set consisting of laminations $\mL'$ with $i(\mL^{\pm},\mL')>\eps$.  Then for any neighborhood $U$ of $\mL^+$ which is disjoint from $K$ and any sufficiently large $N>0$ (which depends on $K$ and $U$), we have $\psi^N(K)\subset U$.
\end{prop}

\subsection{Intersections and freeness}
In this subsection, we will prove a weak version of Theorem \ref{t:main}, which illustrates the fact that intersections of supports of mapping classes give rise to free subgroups of mapping class groups.  Let $f_i$ be a mapping class which is either a pseudo-Anosov homeomorphism supported on a connected subsurface $\Sigma_i$ of $\Sigma$ or a Dehn twist about a simple closed curve $c_i$.  In the first case, $f_i$ has a pair $\mL_i^{\pm}$ of invariant laminations on $\Sigma_i$ and in the second case $f_i$ attracts any lamination which intersects $c_i$ to $c_i$.  In this sense, the laminations $\mL_i^{\pm}$ or $c_i$ can be thought of as the {\bf limiting laminations} of $f_i$.

The following theorem is the main result of this subsection, and is well--known by the work of McCarthy and Ivanov:
\begin{thm}\label{t:intfree}
Let $\{f_1,\ldots,f_k, t_1,\ldots,t_{\ell}\}\subset\Mod_{g,p}$, with each $f_i$ pseudo-Anosov on a subsurface $\Sigma_i$ with limiting laminations $\mL_i^{\pm}$ and each $t_i$ a Dehn twist with limiting lamination $c_i$.  Suppose that for each distinct pair of classes in $\{f_1,\ldots,f_k, t_1,\ldots,t_{\ell}\}\subset\Mod_{g,p}$, the limiting laminations have essential intersection.  Then there exists an $N>0$ such that for all $n\geq N$, we have $\{f_1^n,\ldots,f_k^n,t_1^n,\ldots,t_{\ell}^n\}$ generates a free group.
\end{thm}
\begin{proof}
For each limiting lamination $\mL$ of a mapping class on our list, write $Z_{\mL}$ for the closed subset of $\bP\mathcal{ML}(\Sigma)$ which do not intersect $\mL$.  If $\mL$ is a limiting lamination of any of the mapping classes on the list, then all the other limiting laminations of the other mapping classes are contained in a compact subset $K\subset\bP\mathcal{ML}(\Sigma)\setminus Z_{\mL}$.  The result follows immediately from the discussion in Subsection \ref{s:conv} and the ping--pong lemma.
\end{proof}

We remark briefly that Theorem \ref{t:intfree} is false if we replace $\{f_1,\ldots,f_k, t_1,\ldots,t_{\ell}\}\subset\Mod_{g,p}$ by an arbitrary finite collection of pure mapping classes.  Consider five Dehn twists $T_1,\ldots,T_5$ such that $[T_1,T_2]=1$, $[T_3,T_4]=1$, $[T_1,T_3]=1$, $[T_1,T_4]=1$, $[T_2,T_4]=1$, $[T_2,T_5]=1$ and $[T_3,T_5]=1$, and such that all other pairs of twists do not commute.  Write $\psi_1^m,\psi_2^m,\psi_3^m$ for the mapping classes $(T_1T_2)^m$, $(T_3T_4)^m$ and $T_5^m$ respectively.  For each $m$, we have that \[(\psi_3^m)^{\psi_1^m}=(\psi_3^m)^{(\psi_1^m)^{\psi_2^m}}.\]  Thus, $\{\psi_1^m,\psi_2^m,\psi_3^m\}$ cannot generate a free group of rank three, even though their supports mutually intersect and though the abelianization of the group they generate is isomorphic to $\bZ^3$.

\subsection{Competing convergence}
In this subsection, we discuss the primary technical difficulty which arises when we encounter two disjoint limiting laminations.  Suppose $\mL$ is a lamination which intersects two disjoint curves $c_1,c_2$ with corresponding Dehn twists $T_1$ and $T_2$.  We have seen that in $\bP\mathcal{ML}(\Sigma)$, iterating $T_1$ and $T_2$ on $\mL$ results in convergence to $c_1$ and $c_2$ respectively.  Consider $T_1^N(\mL)$, where $N$ is enormous.  The resulting lamination is very close to $c_1$.  It intersects $c_2$, but very little.  If we start applying $T_2$ to $T_1^N(\mL)$, the lamination starts to converge to $c_2$ but it does so at a very slow rate.  The larger $N$ is, in fact, the slower this convergence is going to be.  This is problematic from the point of view of ping--pong, if we want to look at laminations which are close to the limiting laminations of the mapping classes in question as we do in Theorem \ref{t:intfree}, for instance.

It is better to consider laminations which do not necessarily ``look like" $c_1$ or $c_2$, but which share some of their properties.  A fundamental property we can exploit here is that $i(c_i,c_i)=0$ for $i=1,2$.  For $\eps>0$, we will write $U_{i,\eps}$ for the set of laminations whose weighted intersection with $c_i$ is at most $\eps$.

Suppose that $\gamma$ is a third curve which intersects both $c_1$ and $c_2$ exactly once.  Notice that for any choices of $n$ and $m$, we have that the absolute geometric intersection number \[i_g(c_j,T_1^nT_2^m(\gamma))=1\] for $j=1,2$.  Viewing $\gamma$ as a lamination, we see that as long as the length of $T_1^nT_2^m(\gamma)$ is sufficiently large (in some hyperbolic metric) then the intersection number between the laminations $c_j$ and $T_1^nT_2^m(\gamma)$ is very nearly zero.  We can neatly summarize these observations:

\begin{prop}
For each $N$, write $A_N$ for the group $\langle T_1^N,T_2^N\rangle$.  For any $\eps>0$, any $N$ sufficiently large, and any $1\neq g\in A_N$, we have \[g(\gamma)\in U_{1,\eps}\cap U_{2,\eps}.\]
\end{prop}
\begin{proof}
For simplicity, we will assume that the homology classes of $c_1$ and $c_2$ are linearly independent.  The technical assumptions on the curves $\{c_1,c_2,\gamma\}$ simplify the argument needed to prove the proposition.  In particular, the homology intersection pairing between $\gamma$ and both $c_1$ and $c_2$ is nonzero.  In particular, the homology class of $T_1^nT_2^m(\gamma)$ is given by \[[T_1^nT_2^m(\gamma)]=[\gamma]+n[c_1]+m[c_2].\]  It follows that the word length of $T_1^nT_2^m(\gamma)$ as an element of $\pi_1(\Sigma)$ (which is asymptotically comparable to the hyperbolic length) is at least $n+m$.  It follows that the (lamination) intersection number of $T_1^nT_2^m(\gamma)$ with $c_1$ and $c_2$ tends to zero at least as fast as $1/(n+m)$, which establishes the claim of the proposition.
\end{proof}

We will now generalize the preceding proposition and dispense with some of the assumptions.  Namely, we suppose that $c_1$ and $c_2$ are any two disjoint simple closed curves and $\gamma$ is just some simple closed curve which intersects both $c_1$ and $c_2$.  We will write $n_i$ for the geometric intersection number of $\gamma$ with $c_i$, $T_i$ for the Dehn twist about $c_i$, and $U_{i,\eps}$ for the subset of $\bP\mathcal{ML}(\Sigma)$ as above.  Observe that for any exponents $a$ and $b$, we have that $T_1^aT_2^b(\gamma)$ intersects $c_i$ no more than $n_i$ times.

\begin{prop}\label{p:twocurves}
For each $N$, write $A_N$ for the group $\langle T_1^N,T_2^N\rangle$.  For any $\eps>0$, any $N$ sufficiently large, and any $1\neq g\in A_N$, we have \[g(\gamma)\in U_{1,\eps}\cap U_{2,\eps}.\]
\end{prop}
\begin{proof}
Consider an intersection point of $\gamma$ and $c_i$.  Fixing a hyperbolic metric $\ell$ on $\Sigma$, we may represent $\gamma$ and $c_i$ by geodesics, after which the intersection is necessarily transverse.  Lifting $\gamma$ and $c_i$ to two intersecting geodesics in hyperbolic space, we see that the effect of doing a Dehn twist about $c_i$ is to cut $\gamma$ at the intersection point, translate one piece along $c_i$ by a distance $\ell(c_i)$, and straighten the resulting union of two rays and a segment to a geodesic.  Since $c_1$ and $c_2$ are disjoint, we can perform this cutting and straightening construction independently for the two curves, at each intersection point.

A piecewise geodesic lift of $g(\gamma)$ to $\bH^2$ can be given by following $\gamma$ until we hit an intersection point with a lift of $c_i$, following $c_i$ for the appropriate distance, and then continuing along another lift of $\gamma$.

We claim that there are only a finite number of exponents $e_i$ such that $T_1^{e_1}T_2^{e_2}(\gamma)$ is shorter than any given length.  For any compact subset $K$ of $\bH^2$, there are only finitely many lifts of $\gamma$, $c_1$ and $c_2$ which intersect $K$.  Furthermore, each lift of $c_1$ and $c_2$ to $\bH^2$ cuts $\bH^2$ into two half--spaces.  These half--spaces have the property that whenever our piecewise geodesic representative crosses from one such half--space to another, it never backtracks.  This can be seen easily from the fact that $c_1$ and $c_2$ are disjoint.  Finally, each piece of $\gamma\setminus (\gamma\cap(c_1\cup c_2)$ has length bounded away from zero, since the distance between $c_1$ and $c_2$ on $\Sigma$ is positive.  We see that the piecewise geodesic representative segment for $g(\gamma)$ makes positive progress towards infinity as the exponents for $T_1$ and $T_2$ tend to infinity.  In particular, $g(\gamma)$ can have both of its endpoints in a fixed compact set $K$ for only finitely many different exponents.  It follows that choosing $N$ sufficiently large, for any $1\neq g\in A_N$, we have \[g(\gamma)\in U_{1,\eps}\cap U_{2,\eps},\] since the total number of intersections of $g(\gamma)$ with $c_1$ and $c_2$ does not change but since the length of $g(\gamma)$ tends to infinity.
\end{proof}

The proof of Proposition \ref{p:twocurves} is robust in the sense that it generalizes beyond Dehn twists.  Consider, for instance, two pseudo-Anosov homeomorphisms $\psi_1$ and $\psi_2$ supported on disjoint surfaces with stable laminations $\mL_1$ and $\mL_2$, and a simple closed curve $\gamma$ which intersects both of them essentially.  We can analyze the actions of $\psi_1$ and $\psi_2$ on $\gamma$ in a vein which is similar to the proof of Proposition \ref{p:twocurves}.  One can similarly show that for each $n$, there are only finitely many exponents $e_1$ and $e_2$ for which $\ell(\psi_1^{e_1}\psi_2^{e_2}(\gamma))\leq n$.

\subsection{The issue with boundary twisting}\label{ss:twisting}
In this subsection and for the rest of this section of the paper, we will fix a hyperbolic metric on the surface $\Sigma$, and all laminations under consideration will be assumed to be geodesic.  Let $\psi$ be a pseudo-Anosov mapping class supported on a subsurface $\Sigma'\subset\Sigma$, and let $\beta$ be a boundary component of $\Sigma'$.  Write $\mL^{\pm}$ for the limiting laminations of $\psi$.  It is possible that the laminations $\mL^{\pm}$ contain $\beta$, in which case $\psi$ twists about the boundary component $\beta$.  We call $\beta$ a {\bf peripheral leaf} of $\mL^{\pm}$.

The situation where $\beta\subset\mL^{\pm}$ is pathological for two reasons.  Firstly, various angle estimates which we will be making (see Lemma \ref{l:disjangle}, for instance) will not work, since $\beta$ will not be disjoint from $\mL^{\pm}$.  More importantly, Theorem \ref{t:main} fails when we allow peripheral leaves in pseudo-Anosov laminations.

To see this second claim, we will give an example where we get unexpected commutation.  Let $\psi$ and $\phi$ be pseudo-Anosov mapping classes on $\Sigma'$ which differ by a twist about $\beta$, and let $\varphi$ be another pseudo-Anosov mapping class supported on $\Sigma'$ which does not commute with $\psi$ or $\phi$.  Theorem \ref{t:main} says that for sufficiently large $N$, we have $\langle \psi^N,\phi^N,\varphi^N\rangle\cong \bZ^2*\bZ$, where the first factor is generated by $\{\psi^N,\phi^N\}$.  However, $\psi^{-N}\phi^N$ is different from the identity and commutes with $\varphi$, a contradiction.

Therefore, we will always assume that pseudo-Anosov laminations on subsurfaces of $\Sigma$ have no peripheral leaves, or equivalently that each leaf of $\mL^{\pm}$ is dense.  A pseudo-Anosov mapping class which twists about a boundary component must be thought of as a pseudo-Anosov mapping class which does not twist about the boundary, composed with a boundary twist.

Theorem \ref{t:main} will still hold if some mapping class $\psi$ on $\Sigma'$ with boundary twisting is included, provided that the following assumptions hold:
\begin{enumerate}
\item
No pseudo-Anosov $\phi$ with the same support as $\psi$ virtually commutes with $\psi$.
\item
No other mapping class $\varphi$ is supported on the boundary of $\Sigma'$.
\end{enumerate}

\subsection{Angles of intersection}\label{ss:angle}
In this subsection we will develop some of the necessary technical tools to prove Theorem \ref{t:main}.  The general strategy for the proof of Theorem \ref{t:main} is to play ping--pong on $\bP\mathcal{ML}(\Sigma)$.
The fundamental technical difficulty which arises can be illustrated by the following example: let $x$, $y$ and $z$ be three simple closed curves on $\Sigma$ with $z$ disjoint from both $x$ and $y$ and with $x$ and $y$ intersecting nontrivially.  Suppose $\gamma$ is a simple closed curve which intersects all three of these curves geometrically plenty of times, but in $\bP\mathcal{ML}(\Sigma)$, the lamination $\gamma$ is very close to $z$, so that the intersection pairing of $z$ and $\gamma$ as measured laminations is very small.  Now, perform a moderately--sized power of a Dehn twist about $y$.  How can we control the intersection pairing of the resulting curve with $z$?

The idea is to control the angles at which these various curves intersect.  As observed before, the total number of geometric intersections between $\gamma$ and the three curves in question cannot increase under performing Dehn twists.  If we can guarantee that $\gamma$ stays relatively parallel to $z$ after Dehn twisting about $y$, then the lamination intersection pairing between $T_y^N(\gamma)$ and $z$ cannot increase dramatically.

Consider an intersection point between two geodesics in $\bH^2$.  Since the nonempty intersection of any two distinct geodesics in $\bH^2$ is transverse, the intersection has a well--defined angle which is contained in the open interval $(0,\pi)$.  Of course, whether two geodesics intersect with angle zero or angle $\pi$ is a matter of convention which can be established unambiguously by orienting the geodesics.  To avoid this problem, we will declare the angle of intersection of two geodesics to be the acute angle at the intersection point (or $\pi/2$ if the two angles are the same).  We will call this angle measure the {\bf unoriented angle of intersection}.  We will often say that two geodesics on a surface $\Sigma$ have an unoriented angle of intersection at least $\al>0$.  This means that at least one intersection point between those two geodesics has such an unoriented angle of intersection.  If we say that the minimal unoriented angle of intersection is at least $\al>0$, we mean that each intersection between those two geodesics has angle at least $\al$.

The following result roughly says that different laminations have no parallel leaves.
\begin{lemma}\label{l:lamangle}
Let $\gamma$ be a simple closed geodesic on $\Sigma$ and let $\mL$ be the limiting geodesic lamination of a pseudo-Anosov mapping class $\psi$ on a connected subsurface of $\Sigma$.  There is a $\theta>0$ such that the minimal unoriented angle of intersection between $\gamma$ and any leaf of $\mL$ is at least $\theta$.  The same conclusion holds if $\gamma$ is replaced by a limiting lamination $\mL'$ of another pseudo-Anosov mapping class $\psi'$, provided that $\psi$ and $\psi'$ do not generate a virtually abelian subgroup of $\Mod(\Sigma)$.
\end{lemma}
\begin{proof}
When $\gamma$ is a closed geodesic, this is obvious.  If there were a sequence of intersection points of $\gamma$ with $\mL$ with angle tending to zero then $\gamma$ could not have finite length.

For the case of $\mL'$ and $\mL$, consider the a sequence $\{x_n\}$ of intersection points between leaves whose angles tend to zero at least as fast as $\{1/n\}$.  Note that $\mL'$ and $\mL$ are compactly supported.  So, the sequence $\{x_n\}$ has an accumulation point $x$.  Both $\mL'$ and $\mL$ are closed, so that $x$ is in fact an intersection point between a leaf of $\mL$ and $\mL'$.  It follows that at $x$, there is a leaf of $\mL$ and a leaf of $\mL'$ which coincide.  Since each leaf is dense in both laminations, we see that $\mL=\mL'$.  The lemma follows.
\end{proof}

If $\mL$ in Lemma \ref{l:lamangle} is replaced by a closed geodesic, then the same conclusions obviously hold (provided that $\gamma$ does not coincide with $\mL$).  It is clear that the requirement that pseudo-Anosov laminations have no peripheral leaves is essential.

The following result roughly says that applying powers of a mapping class $f$ to a geodesic $\gamma$ will send intersections angles between $\gamma$ and limiting laminations of $f$ to zero.

\begin{lemma}\label{l:angletozero}
Fix $\eps>0$, and let $\mL^{\pm}$ be the limiting geodesic laminations of a mapping class $f$, which is either pseudo-Anosov on a connected subsurface or of a Dehn twist.  Let $\gamma$ be a geodesic which intersects $\mL^{-}$ at least once with unoriented angle at least $\al>0$.  Then there is an $N>0$ which depends only on $\al$, $\eps$ and $f$, such that for all $n\geq N$, some intersection between $f^n(\gamma)$ and a leaf of $\mL^+$ has unoriented angle of intersection at most $\eps$.
\end{lemma}
The distinction between the intersection of $\gamma$ with $\mL^-$ and with $\mL^+$ only applies when $f$ is pseudo-Anosov.
\begin{proof}[Proof of Lemma \ref{l:angletozero}]
Lift $\mL$ to $\bH^2$ and fix a lift $\yt{\gamma}$ of $\gamma$ and a basepoint $x\in\yt{\gamma}$ which is disjoint from the lift of $\mL$.  Follow $\yt{\gamma}$ from $x$ until we encounter a leaf of the lift of $\mL$.  By choosing $x$ carefully, we may assume that the unoriented angle with which we encounter this first leaf is at least $\al$.

Suppose first that $f$ is a Dehn twist, so that $\mL^{\pm}=\mL$ is just a simple closed curve on $\Sigma$.  The effect of the Dehn twist is to cut open $\yt{\gamma}$, insert a segment along the lift of $\mL$ of length $\ell(\mL)$, and then continue along a translate of $\yt{\gamma}$.  We thus obtain a piecewise geodesic.  In order to find the angle between $f(\gamma)$ and $\mL$, we need to straighten out the piecewise geodesic to be geodesic.  The conclusion of the lemma now follows from the thinness of triangles in $\bH^2$.  A straightened geodesic representing $f(\gamma)$ must stay within a universally bounded distance of the piecewise geodesic which follows $\yt{\gamma}$ from $x$ and then a piece of the lift of $\mL$.  If the inserted piece of the lift of $\mL$ is long enough compared to $\al$, then clearly a geodesic representing $f(\gamma)$ intersects $\mL$ at as small an angle as we wish.  Thus, fixing $\eps$ taking $N$ sufficiently large, we can arrange this intersection to have unoriented angle at most $\eps$.

Of course, $\yt{\gamma}$ will generally intersect infinitely many lifts of $\mL$.  Consider again $\yt{\gamma}$ emanating from $x$ and encountering a lift of $\mL$.  We insert a segment along that lift of $\mL$ and then continue along a translate $\yt{\gamma}'$ of $\yt{\gamma}$.  Call the direction along which we travel on the lift of $\mL$ ``positive", and write $p$ for the corresponding endpoint at infinity in the positive direction.

The translate $\yt{\gamma}'$ of $\yt{\gamma}$ defines a geodesic ray emanating from this lift of $\mL$.  Fixing $x$ within a fundamental domain for $\Sigma$, the point at infinity where this ray hits $\partial \bH^2$ depends only on the length of the inserted geodesic segment along the lift of $\mL$ and the angle $\al$.  By increasing the length of the inserted geodesic segment (which is to say by increasing $N$), we can avoid any finite collection of lifts of $\mL$ in the half--space defined by the one from which the geodesic ray emanates.  Let $\xi$ be the next lift of $\mL$ which we encounter along $\yt{\gamma}'$.  The endpoint of the geodesic lift of $f(\gamma)$ must be contained between the endpoints of $\xi$.  Since all lifts of $\mL$ are disjoint, we can arrange $\xi$ to have endpoints as close to $p$ as we would like by translating further along the first lift of a leaf of $\mL$.  Thus, if $N$ is large enough then the other lifts of $\mL$ will not undo our estimate that each unoriented angle of intersection between $f^N(\gamma)$ and $\mL$ is at most $\eps$.

Now suppose that $f$ is pseudo-Anosov on a subsurface.  The requirement that $\gamma$ has unoriented angle of intersection bounded away from zero with $\mL^-$ precludes $\gamma$ from fellow traveling leaves of the unstable lamination of $f$.  The argument is nearly the same as for Dehn twists, only that the length of the twisting curve is replaced by the dilatation of the pseudo-Anosov.

Pull back $\mL^{\pm}$ to $\bH^2$.  Consider a lift $\yt{\gamma}$ of $\gamma$.  If we fix a basepoint $x$ and follow $\yt{\gamma}$, eventually we will reach leaves of the lift of $\mL^+$ which $\yt{\gamma}$ traverses, and then we will reach a complementary region of $\mL^+$.  The effect of the pseudo-Anosov $f$ on $\gamma$ it by stretching the leaves of the stable lamination by the dilatation and contracting the leaves of the unstable lamination by the dilatation.  If $\gamma$ intersects a leaf of the unstable foliation with an unoriented angle of intersection of at least $\al>0$, then $f$ will stretch $\gamma$ in the direction which is parallel to the leaves of $\mL^+$ at a definite rate which is proportional to the dilatation.  The rest of the proof is the same as in the Dehn twist case.
\end{proof}

\begin{lemma}\label{l:disjangle}
Let $\mL^{\pm}$ be the limiting geodesic laminations of a mapping class $f$ and let $X$ be a limiting geodesic lamination of another mapping class, with $X$ disjoint from $\mL^{\pm}$.  There exists a continuous function \[\Theta:(0,\pi/2]\to (0,\pi/2]\] which depends only on $\mL^{\pm}$ and $X$ such that:
\begin{enumerate}
\item
The value of $\Theta$ tends to zero as the argument tends to zero.
\item
Let $\gamma$ be a geodesic which intersects a leaf of $X$ with unoriented angle of intersection $\al(\gamma,X)$.  Then for any $n$, there is an intersection between $f^n(\gamma)$ and a leaf of $X$ whose angle of intersection $\al(f^n(\gamma),X)$ satisfies \[\al(f^n(\gamma),X)\leq \Theta(\al(\gamma,X)).\]
\end{enumerate}
\end{lemma}
\begin{proof}
Lift all the geodesics in question to $\bH^2$.  Since $\mL^{\pm}$ and $X$ are disjoint, no leaf of $\mL^{\pm}$ can share an endpoint with any leaf of $X$ at infinity.  This is clear for any closed leaves of $\mL^{\pm}$ and $X$ since closed leaves represent homotopy classes of simple closed curves and the discreteness of the representation $\pi_1(\Sigma)\to PSL_2(\bR)$ precludes disjoint closed geodesics from sharing endpoints at infinity.  In general, we can cut $\Sigma$ open along one of the laminations $\mL^{\pm}$ to get a surface $\Sigma'$ of the same area as $\Sigma$ but with geodesic boundary.  For the details of this construction, see \cite{PH}, Proposition 1.6.3.  We then double $\Sigma'$ along this boundary.  The resulting surface contains a copy of $X$ since $\mL^{\pm}$ and $X$ are disjoint, and the image of $\mL^{\pm}$ is a geodesic lamination with finitely many leaves.  These leaves are all either closed geodesics or geodesics connecting cusps in the doubled surface.  Since $X$ is compactly supported on $\Sigma$ and since no leaf of $X$ spirals into any leaf of $\mL^{\pm}$, we have that no leaf of $X$ enters any cusps of the doubled surface.  Repeating the cutting and doubling construction for $X$, we get two laminations, each consisting of finitely many leaves, and each leaf connects cusps or is closed.  It follows easily now that no two leaves of $\mL^{\pm}$ and $X$ share fixed points at infinity when lifted to $\bH^2$.

Fix a lift $\yt{\gamma}$ of $\gamma$ to $\bH^2$ and a leaf $\beta$ of the lift of $X$ which $\yt{\gamma}$ intersects at a point $p$.  We will always assume that this intersection point sits inside of a fixed fundamental domain $K$ for $\Sigma$, which without loss of generality is compact.  Choosing a different point of intersection between different lifts in a different fundamental domain is tantamount to applying an isometry to the entire picture.

If the angle between $\yt{\gamma}$ and $\beta$ is small, then one has to follow $\yt{\gamma}$ for a very long distance before encountering a leaf of the lift of $\mL^{\pm}$.  Let $q_1$ be the first leaf of  the lift of $\mL^{\pm}$ which $\yt{\gamma}$ encounters.  Observe that for any $n$, the endpoint at infinity of $f^n(\yt{\gamma})$ on the side of $\beta$ which contains $q_1$ must lie between the endpoints of $q_1$.  If the unoriented angle of intersection $\al$ at $p$ decreases beyond some threshold then $\yt{\gamma}$ will avoid $q_1$ altogether, since $q_1$ and $\beta$ share no endpoints at infinity.  Then, the next candidate $q_2$ for the first leaf of the lift of $\mL^{\pm}$ which $\yt{\gamma}$ encounters is a geodesic whose endpoints are wedged between an endpoint of $q_1$ and an endpoint of $\beta$ (since $\yt{\gamma}$ did not encounter $q_2$ first).  Since $\beta$ and $q_2$ share no endpoints at infinity, we can repeat the same argument.  It follows that as $\al$ tends to zero, the set candidates for endpoints of the first leaf of the lift of $\mL^{\pm}$ encountered by $\yt{\gamma}$ as it travels away from $p$ makes up a smaller and smaller subset of the boundary of $\bH^2$ at infinity, in the visual metric as seen from $p$.

The observations of the previous paragraph also hold on the other side of $\beta$, which is to say as $\gamma$ moves away from $p$ in the other direction.  It follows that for any $n$, the angle between $f_1^n(\yt{\gamma})$ and $\beta$ cannot increase or decrease drastically.

It is clear that the function $\Theta$ above can be constructed to be continuous, since it is just measuring the value to which the intersection of $\yt{\gamma}$ and $\beta$ can increase as we move endpoints of $\yt{\gamma}$ between endpoints of lifts of leaves of $\mL^{\pm}$.
\end{proof}

\subsection{Fellow traveling}
Let $\gamma(t)$ and $\gamma'(t)$ be two unit speed parametrized geodesics in $\bH^2$ which intersect at a point $p$ at $t=0$.  It is well--known that as we flow away from time $0$, the distance between $\gamma(t)$ and $\gamma'(t)$ diverge quickly.  If the angle of intersection at $p$ in the direction of positive time travel is small, then $\gamma(t)$ and $\gamma'(t)$ will stay near each other, at least for a short while.

\begin{lemma}\label{l:fellowtravel}
Let $\mL$ be a simple closed geodesic or a limiting geodesic lamination of a pseudo-Anosov homeomorphism supported on a connected subsurface of $\Sigma$.  Let $X$ be a geodesic which intersects $\mL$ with minimal unoriented angle of intersection $\al>0$.  For each $\al'<\al$, there is a $\theta(\al')>0$ which depends only on $\mL$, $X$ and $\al'$ such that if the unoriented angle of intersection between a geodesic $\gamma$ and $\mL$ is at most $\theta$ then $\gamma$ intersects $X$ with unoriented angle at least $\al'$.
\end{lemma}
\begin{proof}
Lift the entire setup to the universal cover $\bH^2$.  Suppose first that $\mL$ is a simple closed curve of length $\ell$.  Then each component of the lift of $\mL$ intersects a component of the lift of $X$, and the maximum distance between two such intersections along one lift of $\mL$ is $\ell$.  Let a lift $\yt{\gamma}$ of $\gamma$ intersect a lift of $\mL$.  Traveling along $\mL$, we will encounter a lift of $X$ at an unoriented angle of $\al$ at most a distance $\ell$ away.  If the unoriented angle between the lift of $\mL$ and $\yt{\gamma}$ is small enough, $\yt{\gamma}$ will intersect a lift of $X$ with unoriented angle nearly $\al$.

Now suppose that $\mL=\mL^+$ is the limiting lamination of a pseudo-Anosov on a subsurface $\Sigma'$.  Parametrize the leaves of $\mL$ to have speed one.  Note that if $T$ is a transversal to the lamination then the first return time of leaves of $\mL$ to $T$ is bounded.  This follows from several well--known facts: the first is that $\mL$ can be transformed into a singular foliation $\mF^+$ by collapsing complementary components of $\mL$.  Together with the singular foliation $\mF^-$ associated to $\mL^-$, these foliations give $\Sigma'$ a flat metric which is comparable to the hyperbolic metric (cf. \cite{FLP} for instance).  Passing to a double branched cover if necessary, we may assume $\mF$ is orientable.  Then, the first return map for a transversal is just an interval exchange map and thus has bounded first return time.  This establishes the claim.

We now obtain the conclusion of the lemma, since then we can take a neighborhood of an intersection point of $X$ and $\mL$ to be a transversal.  If $\gamma$ follows a leaf of $\mL$ closely for the maximum first return time to the transversal, then $\gamma$ will intersect the transversal as well.  By taking $\theta$ small enough, we can make the unoriented angle of intersection as close to the unoriented angle of intersection between $X$ and $\mL$ as we like.
\end{proof}

\subsection{Greedy normal forms in right-angled Artin groups}
Before we give the proof of Theorem \ref{t:main}, we will describe a normal form which we will use in the course of the proof.

Let $w\in A(\gam)$ be a reduced word.  We will rewrite $w$ as a product \[w=w_kw_{k-1}\cdots w_1\] of central words in a certain way.  We will want no generator which occurs in $w_i$ to commute with $w_{i+1}$.  To achieve this, first we require that $w_1$ be as long as possible, subject to the condition that it still be a central word.  Then, we arrange for $w_2$ to be as long as possible as a subword of $w\cdot w_1^{-1}$, subject to the requirement that it be central.  We continue in this fashion until we obtain an expression for $w$ as a product of central words.

Now suppose that a generator $v$ which occurs in $w_{k-1}$ commutes with $w_k$.  We move all occurrences of $v$ to $w_k$, thus shortening $w_{k-1}$.  Repeat this moving of generators from $w_{k-1}$ to $w_k$ until each generator occurring in $w_{k-1}$ does not commute with $w_k$.  If $w$ is not itself a central word, $w_{k-1}$ will be nontrivial at the end of this process.  We now repeat this process for $w_{k-2}$.  If $v'$ is a generator occurring in $w_{k-2}$ which commutes with $w_{k-1}$ then move each occurrence of $v'$ to $w_{k-1}$.  If $v'$ commutes with $w_k$ as well, move the occurrences of $v'$ in $w_{k-2}$ to $w_k$.  In general, if we have a generator in $w_i$ which commutes with $w_{i+1}$, we move all occurrences of this generator as far to the left as possible.  This rewriting process will terminate in finite time.  Note that after we move generators occurring in $w_i$ as far to the left as possible, we may move new generators to $w_i$ from some $w_j$ with $j<i$.  However, the generators of $w_i$ which were there already will not be moved to the right.  The resulting expression of $w$ as a product of central words will be called {\bf left greedy normal form}.  For our purposes, no uniqueness of the left greedy normal form is required -- only that it exists.

\subsection{Playing ping--pong and the proof of Theorem \ref{t:main}}
Let $\{f_1,\ldots,f_k\}$ be the mapping classes in question.  We assume that this collection of mapping classes is irredundant.  We write \[L=\{\mL_1^{\pm},\ldots,\mL_k^{\pm}\}\] for the invariant laminations of these mapping classes, which are all either pseudo-Anosov laminations supported on connected subsurfaces of $\Sigma$ in which case $\mL_i^+$ and $\mL_i^-$ are different, or simple closed curves, in which case $\mL_i^+=\mL_i^-$.  A subcollection of $L$ whose elements are all disjoint is called an {\bf admissible configuration}.  An admissible configuration $\mC$ naturally gives rise to an embedded simplex $\Delta_{\mC}$ inside of $\bP\mathcal{ML}(\Sigma)$ by taking non--negative total weight one linear combinations of the elements in the subcollection.

\begin{proof}[Proof of Theorem \ref{t:main}]
All laminations and curves under consideration are geodesic.  For any pair of distinct, intersecting laminations $\mL_i$ and $\mL_j$ on the list $L$, write $\eps_{i,j}$ for the minimal unoriented angle of intersection between leaves of $\mL_i$ and $\mL_j$.  Notice that if $f_i$ is pseudo-Anosov on a subsurface, then we include the minimal unoriented angle of intersection between $\mL_i^+$ and $\mL_i^-$.  Let $\gamma$ be a simple closed curve which intersects each element of $L$ essentially (and is therefore not on the list $L$).  Since $L$ is a finite set, there is a positive lower bound $\eps_{\gamma}$ on the minimal unoriented angle of intersection between $\gamma$ and leaves of elements of $L$.  Write \[\delta=\frac{\min(\eps_{i,j},\eps_{\gamma})}{2}.\]  The value of $\delta$ is positive by Lemma \ref{l:lamangle}.

Let $\eps>0$, which we will assume for the rest of the proof is smaller than $\delta$.  By Lemma \ref{l:angletozero}, there is an exponent $N=N(\delta)>0$ such that if $\gamma$ is any geodesic which intersects the limiting laminations $\mL^{\pm}$ of some $f_i$ with unoriented angle of intersection at least $\delta$, then some unoriented angle of intersection between $\mL_i^{+}$ and $f_i^n(\gamma)$ is at most $\eps$ for all $n\geq N$.  Suppose that instead of $f_i^N(\gamma)$ we consider \[f_i^ng_1^{n_1}\cdots g_j^{n_j}(\gamma),\] where each one of $\{g_1,\ldots,g_j\}$ is equal to a mapping class on the list $\{f_1,\ldots,f_k\}$ and all the $\{g_1,\ldots,g_j\}$ commute with each other and with $f_i$.   In this case, Lemma \ref{l:disjangle} gives us a function $\Theta$ such that the unoriented angle of intersection between \[f_i^n g_1^{n_1}\cdots g_j^{n_j}(\gamma)\] and $\mL_i^+$ will be at most $\Theta^j(\eps)$, where by $\Theta^j$ we mean $\Theta$ composed with itself $j$ times.  We take $\Theta$ here to be the minimum of the functions furnished by Lemma \ref{l:disjangle} which work for pairs of disjoint laminations in $L$.

For any fixed $\eps>0$ and $j$, and for all sufficiently small $\eps_0$, we have $\Theta^j(\eps_0)<\eps$.  The maximal size of a commuting collection of mapping classes on the list $\{f_1,\ldots,f_k\}$ is bounded, so that $j$ is always bounded.  Therefore, given $\eps_0>0$, we can choose $N$ sufficiently large so that if $\{f_i, g_1,\ldots,g_j\}$ are a commuting collection in $\{f_1,\ldots,f_k\}$, we have that the unoriented angle of intersection between $\mL_i$ and \[f_i^n g_1^{n_1}\cdots g_j^{n_j}(\gamma)\] is at most $\eps$, whenever $n\geq N$.

In the first paragraph, we fixed a simple closed curve $\gamma$ which intersects each element of $L$.  By replacing the $N(\delta)$ in the previous paragraph by a possibly larger exponent, we can arrange that the unoriented angle of intersection between $f_i^n(\gamma)$ and $\mL_i^+$ is at most $\eps_0$.

Recall that in Lemma \ref{l:fellowtravel}, we showed that if a geodesic $\gamma$ follows some $\mL$ closely enough then $\gamma$ will intersect any other $\mL'$ which $\mL$ intersects, with unoriented angle close to that of $\mL$ and $\mL'$.  For each pair of laminations associated to mapping classes on the list $\{f_1,\ldots,f_k\}$, there is a $\theta_0>0$ such that if $\gamma$ intersects $\mL$ with unoriented angle at most $\theta_0$ then $\gamma$ will intersect $\mL'$ with unoriented angle at least $\delta$.  We will be insisting that $\eps<\theta_0$.  Henceforth in the proof, fix $\eps$ and $\eps_0$ which satisfy the conditions we have specified above.

Let $\gam$ be the coincidence graph of the limiting laminations of elements of $L$, so that there is a vertex for each pair $\mL_i^{\pm}$ and an edge between $\mL_i^{\pm}$ and $\mL_j^{\pm}$ if and only if $\mL_i^{\pm}$ and $\mL_j^{\pm}$ are disjoint.  Labeling the vertices of $\gam$ by $\{v_1,\ldots,v_n\}$, for each $N$ we obtain a homomorphism \[\phi_N:A(\gam)\to\Mod_{g,p}\] which sends $v_i$ to $f_i^N$.  We will show that for all $N$ sufficiently large, this homomorphism is injective.  For each $N$, we will not distinguish notationally between elements of $A(\gam)$ and their image in $\Mod_{g,p}$ under $\phi_N$.  The particular $N$ we will be choosing is the one which makes Lemma \ref{l:angletozero} work for $\{f_1,\ldots,f_k\}$, with $\al=\delta$ and $\eps_0$ as above.

Let $U^{+}_{i,\eps},U^-_{i,\eps}\subset\bP\mathcal{ML}(\Sigma)$ be the laminations on $\Sigma$ such that if $c\in U^+_{i,\eps}$ then there is a leaf of $c$ which intersects a leaf of $\mL_i^+$ with unoriented angle at most $\eps$, and respectively for $U^-_{i,\eps}$ and $\mL_i^-$.  For each admissible configuration $\mC$, we will write \[U_{\mC,\eps}=\bigcap_{\mL\in\mC}U_{\mL,\eps},\] where $U_{\mL,\eps}=U^{\pm}_{i,\eps}$, depending on whether $\mL=\mL_i^{+}$ or $\mL_i^-$.

We can now finish the proof of the result.  The proof will use a version of ping--pong and the expression of each word $w\in A(\gam)$ in left greedy normal form.  For the base case, let $w$ be a central word.  Write $\mC$ for the admissible configuration which is the support of $w$, viewed as a mapping class.  Let $\al$ be the minimal unoriented angle of intersection between $\gamma$ and $\mL_i^{\pm}$, as $i$ is allowed to vary.  Since $\eps$ was smaller than $\delta$ and since $\al>\delta$, we have that \[\gamma\notin \bigcup_i U_{i,\eps}^{\pm}\] for any $i$.

We have already shown in the second paragraph that $w(\gamma)$ has unoriented angle of intersection at most $\eps$ with each component of $\mC$.  In particular, $w(\gamma)\in U_{\mC,\eps}$.

Now suppose $c\in U^{\pm}_{i,\eps}$ and that $f_j$ does not commute with $f_i$.  By Lemma \ref{l:fellowtravel}, we may assume that $c$ intersects $\mL_j^{\pm}$ with angle at least $\delta$.  In particular, $f_j^n(c)\in U^+_{j,\eps_0}$ for all $n\geq N$ (and similarly $f_j^{-n}(c)\in U^-_{j,\eps_0}$ for all $n\geq N$).

We claim that this proves the result.  Indeed, write $w=w_k\cdots w_1$ in left greedy normal form.  We have that $w_1(\gamma)\in U^{\pm}_{i,\eps}$ for some $i$.  We may assume that $w_{k-1}\cdots w_1(\gamma)\in U^{\pm}_{j,\eps}$, that that there is a vertex generator occurring in $w_k$, corresponding to the mapping class $f_m$, which does not commute with $f_j$.  Any curve in $U^{\pm}_{j,\eps}$ intersects $\mL_m^{\pm}$ with unoriented angle at least $\delta$, so that combining Lemmas \ref{l:angletozero} and \ref{l:disjangle}, we see that for any $n\geq N$, \[f_m^n(w_{k-1}\cdots w_1(\gamma))\in U^{+}_{m,\eps_0},\] so that \[w(\gamma)=w_k(w_{k-1}\cdots w_1(\gamma))\in U^{+}_{m,\eps},\] and similarly after replacing $f_m^n$ by a negative power of $f_m$, and replacing $U^{+}_{m,\eps}$ by $U^{-}_{m,\eps}$.  In particular, $w(\gamma)\neq \gamma$.
\end{proof}

\section{Right-angled Artin groups in $\Mod_{g,p}$ and the isomorphism problem for subgroups of $\Mod_{g,p}$}
Recall the precise formulation of the isomorphism problem for a class of finite presented groups (see Miller's book \cite{Mil} for a classical account of decision problems in group theory):  We are given a recursive class $\Phi=\{\Pi_i\}$ of finitely presented groups, and we wish to find an algorithm $\phi$ which on the input $(i,j)$ determines whether or not there exists an isomorphism between $\Pi_i$ and $\Pi_j$.  It is an open problem to solve the isomorphism problem for finitely presented subgroups of $\Mod_{g,p}$ (see Farb's article \cite{Farb}).  Usually, it is required that $\Phi$ be a class of finitely presented groups, though in our setup it will be more natural to require that they be merely finitely generated in our setup.

\subsection{Decision problems in $\Mod_{g,p}$}
Let $A\subset\Mod_{g,p}$ be a finite generating set.  We may recursively enumerate words in the free group $F(A)$ as $\{a_1,a_2,\ldots\}$, by choosing an ordering on $A\cup A^{-1}$ and listing words in lexicographical order and by increasing length, for instance.  Since $\Mod_{g,p}$ has a solvable word problem with respect to any finite generating set, we may omit any $a_i$ which is equal in $\Mod_{g,p}$ to some $a_j$ with $j<i$.  We may thus recursively enumerate all finitely generated subgroups of $\Mod_{g,p}$: for each $n$, enumerate all subsets of $A_n=\{a_1,\ldots,a_n\}$, omitting those which were enumerated for some $m<n$.  We then associate to each finite subset of $A_n$ the group generated by those elements.

\begin{prop}\label{p:unsolvable}
There is no algorithm which determines whether or not two finitely generated subgroups of $\Mod_{g,p}$ are isomorphic.
\end{prop}

This proposition was known by the work of Stillwell in \cite{Still} in the context of the {\bf occurrence problem}, also known as the {\bf membership problem}.  The membership problem was originally formulated by Mihailova in \cite{Mi}.  The membership problem seeks to determine whether a given element $w\in G$ is contained in a subgroup $\langle F\rangle<G$ generated by a finite set $F$.  Related discussion can be found in Hamenst\"adt's article \cite{Ham}.

Proposition \ref{p:unsolvable} follows from a more general result about groups which contain sufficiently complicated products of free groups.  If $G$ is any finitely generated group, we can recursively enumerate the elements of $G$, and hence all finitely generated subgroups of $G$ as for $\Mod_{g,p}$.

We will be relying fundamentally on the following well--known result (see for instance Lyndon and Schupp's book \cite{LySch}, or Miller's book \cite{Mil}):
\begin{lemma}
$G=F_n\times F_n$ has an unsolvable generation problem whenever $n\geq 6$.  Precisely, it is undecidable whether or not a finite subset $F\subset G$ generates all of $G$.
\end{lemma}

The following result was claimed in the introduction:
\begin{lemma}\label{l:f9}
There is no algorithm which determines whether two finitely generated subgroups of $F_2\times F_2$ are isomorphic.
\end{lemma}

Lemma \ref{l:f9} is well--known, and a proof can be found in \cite{Mil2}, for instance.  It is a theorem of Bridson and Miller (see \cite{BridMill}) that there is no algorithm to determine the homology of an arbitrary finitely generated subgroup of $F_2\times F_2$.  On the other hand, they prove that the isomorphism problem for finitely presented subgroups is solvable.  We remark that G. Baumslag and J.E. Roseblade showed in a precise sense that ``most" finitely generated subgroups of a product of two nonabelian free groups are not finitely presented (see \cite{BaumRose}).  On a related theme, G. Levitt has proven the unsolvability of the isomorphism problem for free abelian--by--free groups (see \cite{Lev}).

\begin{proof}[Proof of Lemma \ref{l:f9}]
Whenever we produce a generating set for $F_2\times F_2$, we will suppose that it consists of a generating set for $F_2\times \{1\}$ and a generating set for $\{1\}\times F_2$.  Once we have such a generating set, it will be easy to produce a generating set for $F_6\times F_6$ sitting inside of $F_2\times F_2$, using the Nielsen--Schreier rewriting process (see \cite{LySch}).

Suppose the existence of an algorithm which determines whether two finitely generated subgroups of $F_2\times F_2$ are isomorphic.  Using Stallings' folding automaton (see \cite{St}), we can determine whether a given finite subset of $F_2\times F_2$ with respect to our nice generating set sits inside of $F_6\times F_6$.  We will show that this algorithm determines whether a finite subset of $F_6\times F_6$ generates all of $F_6\times F_6$, whence the conclusion of the lemma.  Given a finite subset $S$ of $F_6\times F_6$, the algorithm returns whether or not $\langle S\rangle<F_6\times F_6$ is isomorphic to $F_6\times F_6$.    If the algorithm returns ``no" then $S$ clearly does not generate all of $F_6\times F_6$.

Suppose that the algorithm returns ``yes", so that we have $G=\langle S\rangle\cong F_6\times F_6$.  Decompose $G$ as $G_1\times G_2$, where each $G_i\cong F_6$.  Let $p_1$ and $p_2$ denote the projections of $F_6\times F_6$ onto the two factors.  Since $G$ is included into $F_6\times F_6$, we may suppose that $p_1(G_1)$ is nonabelian.  Note that centralizers of nonidentity elements inside a nonabelian free group are cyclic.  Therefore if $p_1(G_1)$ is nonabelian, we have $p_1(G_2)$ is trivial.  It follows that $p_2(G_2)\cong G_2$, so that $p_2(G_1)$ is trivial.  In particular, $p_1(G_1)\cong G_1$.

Thus, we see that $S$ generates all of $F_6\times F_6$ if and only if the inclusions of $p_i(G_i)$ (which we just denote by $G_i$) into the factors are surjective.  Elements of $S$ are written in terms of generators for $F_6\times \{1\}$ and $\{1\}\times F_6$, so that $G_i$ is obtained simply by deleting the generators for one of the factors.  We are thus given a finitely generated subgroup of $F_6$ and asked to determine whether it is all of $F_6$.  It is well--known that this problem is effectively solvable (again by Stallings' folding automaton, for instance).  It follows that we can algorithmically determine whether or not $S$ generates a proper subgroup of $F_6\times F_6$, a contradiction.
\end{proof}

The proof of Proposition \ref{p:f2f2} is almost immediate now.
\begin{proof}[Proof of Propostion \ref{p:f2f2}]
Suppose that we are given $G$ and a copy $F_2\times F_2<G$.  Restricting the isomorphism problem to finitely generated subgroups of the copy of $F_2\times F_2$, we see that if we could solve the isomorphism problem for finitely generated subgroups of $G$ then we would be able to solve it for $F_2\times F_2$, a contradiction of Lemma \ref{l:f9}.
\end{proof}

Note that in the proof of Proposition \ref{p:f2f2} we are assuming that we are handed a copy of $F_2\times F_2$ inside of $G$.  This does not detract from the generality of the discussion.  Indeed, suppose we knew that $G$ contains a copy of $F_2\times F_2$ but we do not know how to immediately identify it.  If the claim ``the isomorphism problem for finitely generated subgroups of $G$ is unsolvable" were to be contested, then one would need to produce an algorithm which solves the problem.  But then one could enumerate finitely generated subgroups of $G$ and find a copy of $F_2\times F_2$ in finite time.  Then, one would be reduced to the situation of having a group $G$ and a copy of $F_2\times F_2$ identified inside of $G$.

\begin{proof}[Proof of Proposition \ref{p:unsolvable}]
By Lemma \ref{l:f9}, we only need to show that we can effectively identify copies of $F_2\times F_2$ inside of $\Mod_{g,p}$.  Clearly if we have four distinct, essential simple closed curves $\{c_1,\ldots,c_4\}$ in $\Sigma$ such that $\{c_1,c_2\}$ and $\{c_3,c_4\}$ are pairwise disjoint and within both $\{c_1,c_2\}$ and $\{c_3,c_4\}$ we have intersections, then the squares of Dehn twists about these curves will generate a copy of $F_2\times F_2$.  Most presentations of the mapping class group are given in terms of Dehn twists about simple closed curves, so it is usually easy to find such a copy of $F_2\times F_2$.

If one is given a more complicated generating set, one can add a collection of twists about four curves in the configuration described above into the generating set.  It can be shown that the solution to the isomorphism problem for finitely generated subgroups does not depend on the given presentation of the ambient group.  To see this easily in our case, the paragraph following the proof of Proposition \ref{p:f2f2} illustrates that if a group $G$ even abstractly contains a group which provides a certificate of the unsolvability of the isomorphism problem for finitely generated subgroups (such as $F_2\times F_2$) then changing the presentation of $G$ will not change the solvability of the isomorphism problem.
\end{proof}

\subsection{Homological rigidity of right-angled Artin groups: no accidental isomorphisms}
Let $\gam$ and $\gam'$ be two finite graphs.  If $\gam\neq\gam'$ then the standard presentations of the associated right-angled Artin groups $A(\gam)$ and $A(\gam')$ are different, but it is not obvious that $A(\gam)$ and $A(\gam')$ are not isomorphic.  We now prove the following rigidity theorem, which effectively solves the isomorphism problem for right-angled Artin groups.

The following is Theorem \ref{t:rigidhomo} from the introduction.
\begin{thm}
Let $G$ and $G'$ be two right-angled Artin groups.  Then $G\cong G'$ if and only if $H^*(G,\bQ)\cong H^*(G',\bQ)$ as algebras.
\end{thm}
Here we drop the notation ``$A(\gam)$" for a right-angled Artin group in order to emphasize the fact that we do not know the underlying graph.

We will show that if we know the cup product structure on the cohomology algebra of $G$, then we can reconstruct a unique graph $\gam$ for which $G\cong A(\gam)$.  We will actually prove a slightly stronger statement: $\gam$ is determined by $H^1(G,\bQ)$, $H^2(G,\bQ)$, and the cup product \[\cup:H^1(G,\bQ)\otimes H^1(G,\bQ)\to H^2(G,\bQ).\]  It is not surprising that such a strengthening holds since a right-angled Artin group is determined only by a graph, not by an entire flag complex.

Thus, if we are given the data of the abstract isomorphism type of a right-angled Artin group $G$ which is secretly the right-angled Artin group $(G,S)$, we can use the cohomology algebra of $G$ to very nearly recover $S$, and indeed nearly enough to recover the underlying graph.  The isomorphism type of the graph is uniquely determined, so that we recover the main result of \cite{Droms2}.  In the end, we will show that if $(G,S)$ is a right-angled Artin group and $G'$ is abstractly isomorphic to $G$ then there is an essentially unique way to find a right-angled Artin system $S'$ such that we obtain an isomorphism of pairs $(G,S)\cong (G',S')$.

\begin{proof}[Proof of Theorem \ref{t:rigidhomo}]
Throughout the proof we will be tacitly identifying vertices of a graph $\gam$ with their Poincar\'e duals in $H^1(A(\gam),\bZ)$.  Note that Poincar\'e duals of vertex generators do make sense since each clique corresponds canonically to a torus within the Salvetti complex.

Suppose that $\gam$ is a finite graph with vertex set $V$ which is a candidate for a graph which gives rise to $G$.  We claim that the structure of the cohomology algebra allows us extract a basis $X$ for $H_1(G,\bQ)$ such that there exists a map of sets $X\to V$ which induces an isomorphism of algebras $H^*(G,\bQ)\to H^*(A(\gam),\bQ)$.

There is a (possibly trivial) subspace $W_0\subset H^1(G,\bQ)$ which is an isotropic subspace for the cup product.  Its dimension is precisely the number of isolated vertices of $\gam$.  Choose an arbitrary basis $X_0$ for $W_0$.  The cup product descends to the quotient space $A=H^1(G,\bQ)/W_0$.

Consider the vectors $\{v\}$ in $A$ such that the map \[f_v:A\to H^2(G,\bQ)\] given by $f_v(c)=v\cup c$ has rank exactly one.  We think of $v$ as a weighted sum of vertices of $\gam$.  Clearly $v$ cannot be supported on a vertex of degree more than one, since otherwise the rank of $f_v$ would be too large.  Suppose that $v$ is supported on at least two vertices of degree one.  Then all these vertices must in fact be connected to one single vertex, by rank considerations.  Precisely, if $\{v_1,\ldots,v_k\}$ is the support of $v$, then for all $i,j$, we have $\lk(v_i)=\lk(v_j)=v_{k+1}$ for some other vertex $v_{k+1}$ which is not contained in the support of $v$.

Choose a basis $B_1$ for the span of the cohomology classes $v$ such that $f_v$ has rank one, and suppose that each element $b$ in this basis satisfies $\rk f_b=1$.  We may first choose basis vectors which cup trivially with all the other rank one cohomology classes, and these correspond to degree one vertices which are not connected to any other degree one vertices.  Since the remaining degree one vertices pair off as disconnected components of the underlying graph, the choice of the other basis vectors is largely irrelevant.

To interpret the rank of $f_v$ properly, suppose that $v$ is supported on vertices $v_1,\ldots,v_m$.  Note first that the rank of $f_v$ is at most the number of edges in the union of the stars of $v_1,\ldots,v_m$, but we shall see that in general there is an inequality.

Suppose $v$ is a cohomology class with $\rk f_v=n>1$.  Clearly $v$ is not supported on any vertex of degree greater than $n$.  Suppose that $v$ is supported on multiple vertices of degree $n$.  If $v_1$ and $v_2$ are two such vertices, we have $n$ edges emanating out of each of $v_1$ and $v_2$.  Note that there can be an overlap of at most one edge between these two collections of edges, since loops of length $2$ are not allowed.  Thus there are at least $2n-1$ distinct edges emanating from $v_1$ and $v_2$.  Let $w_1,\ldots,w_n$ be the vertices adjacent to $v_1$, and let $z_1,\ldots,z_n$ be the vertices adjacent to $v_2$.  Note that if there is a vertex $z$ which is adjacent to $v_2$ but not in the star of $v_1$ then the rank of $f_v$ is already too large.  It follows that \[\St(v_1)\cup v_2=\St(v_2)\cup v_1.\]  In this case we say that $v_1$ and $v_2$ have {\bf essentially the same star}.

Suppose that $\{v_1,\ldots,v_k\}$ are classes in $H^1(G,\bQ)$ whose ranks (under the map $v\mapsto f_v$) are exactly $n$, and assume that each $v_i$ is supported on at least one vertex of rank $n$.  Suppose furthermore that every nonzero linear combination of $\{v_1,\ldots,v_k\}$ also has rank exactly $n$ and that the total linear span of $\{v_1,\ldots,v_k\}$ has dimension $k$.  It follows that from each pair $\{v_i,v_j\}$ in the collection of classes, we can produce a cohomology class which has rank $n$ and is supported on two distinct vertices of degree $n$.  It follows by the observation above that all the degree $n$ vertices which are in the support of $\{v_1,\ldots,v_k\}$ have essentially the same star.

Inductively, choose a basis $B_i$ for the span of cohomology classes $v$ for which $\rk f_v=i$, for each $i$.  We start first with rank one classes.  Suppose we have chosen a basis for the rank $n$ classes.  We choose a linearly independent set of rank $n+1$ classes.  If a given rank $n+1$ class is not in the span $W_i$ of \[\bigcup_{i\leq n} B_i,\] we first see if it is contained in a subspace $V\subset H^1(G,\bQ)$ of dimension at least $2$ and consisting of rank $n+1$ classes.  If a given class is contained in such a subspace $V$ and if $V\cap W_i=\{0\}$, we see that there are at least two vertices with essentially the same star.  If we can find such a subspace $V$, we can choose a basis for it arbitrarily.  If no such subspace exists for a given rank $n+1$ cohomology class $v$, then $v$ is supported on a unique vertex $w$ of degree $n+1$ and possibly also on some degree one vertices adjacent to $w$.  In this sense, we can say that $w$ is determined up to {\bf degree one vertex ambiguity}.

It is now easy to reconstruct the graph underlying $G$.  The basis vectors we chose either correspond to unique vectors up to degree one ambiguity, or they naturally come in bunches which correspond to vertices with essentially the same star.  Edges of the graph can be reconstructed by cupping basis vectors together.
\end{proof}

\begin{cor}
$A(\gam)\cong A(\gam')$ if and only if $\gam\cong\gam'$.
\end{cor}

If $G$ is a group, we say that $G$ has a solvable {\bf generalized isomorphism problem} for finitely generated subgroups if there is an algorithm $A$ and an irredundant list $L$ such that for every finite $\{f_1,\ldots,f_n\}=F\subset G$ there is an $N=N(F)$ such that $A$ takes as input the data of $F$ and returns an element of $L$ which is the isomorphism type of $\{f_1^N,\ldots,f_n^N\}$.  By Theorem \ref{t:main}, we thus obtain the following solution to the generalized isomorphism problem for finitely generated subgroups of the mapping class group:
\begin{cor}
Let $\{f_1,\ldots,f_k\}\subset \Mod_{g,p}$ be powers of Dehn twists about simple closed curves or pseudo-Anosov on connected subsurfaces.  Then there exist an $N$ such that for all $n\geq N$, the isomorphism type of $\langle\{f_1^n,\ldots,f_k^{n}\}\rangle$ is determined by the intersection pattern of the canonical reduction systems $\{\mathcal{C}_i\}$ for each $f_i$, together with the data of whether $f_i$ and $f_j$ virtually commute on each component of $\mathcal{C}_i$, $\mathcal{C}_j$, $\Sigma\setminus\mathcal{C}_i$, and $\Sigma\setminus\mathcal{C}_j$.
\end{cor}

In particular, even though the isomorphism problem for finitely generated subgroups of the mapping class is unsolvable, it becomes solvable for a large class of subgroups if we are willing to replace elements of a finite generating set by sufficiently large positive powers.  For arbitrary finitely generated subgroups, we would need the following to hold in order to determine the isomorphism type of the subgroup generated by large powers of the corresponding generators:

\begin{conje}
The isomorphism problem for groups enveloped by right-angled Artin groups is solvable.
\end{conje}

We briefly remark that one subtlety in Theorem \ref{t:main} was that we insisted on irredundancy of the collection $\{f_1,\ldots, f_k\}$.  This does not pose great difficulty for us for the following reason.  Suppose that two elements $f_i,f_j$ on our list of mapping classes were equal.  Then we could algorithmically determine whether or not they are equal, since the mapping class group has a solvable word problem.

In the setup of Theorem \ref{t:main}, irredundancy is equivalent to $\langle f_i,f_j\rangle\neq\bZ$ whenever $i\neq j$, which is somewhat more subtle question than just equality.  To determine this purely from the algebra of $\Mod_{g,p}$, one needs to show the existence of a mapping class $g$ which commutes with $f_i$ but not $f_j$, which is nontrivial from a decision problem standpoint.  Since we are already insisting that the mapping classes we consider be of a particular geometric form, we simply require that our list $L$ be irredundant.

\subsection{Embedding right-angled Artin groups in mapping class groups}
We now turn our attention to the question of which right-angled Artin groups appear in a given mapping class group.  Such problems have been discussed in \cite{CW}, for example.

\begin{quest}
Let $\gam$ be a finite graph with vertex set $V$, $A(\gam)$ the associated right-angled Artin group, and let $\Mod_{g,p}$ be a mapping class group.  Under what conditions is there an embedding $A(\gam)\to\Mod_{g,p}$?  What types of mapping classes can we arrange in the image of $V$?
\end{quest}

There are some easy preliminary observations we can make:

\begin{lemma}
Let $\gam$ be a finite graph with vertex set $V$, let $v\in V$, and let $k\neq 0$.  Then the subgroup of $A(\gam)$ generated by $v^k$ and $V\setminus v$ is isomorphic to $A(\gam)$.
\end{lemma}
\begin{proof}
Let $X$ be a set on which $A(\gam)$ plays ping--pong, together with subsets $\{X_i\}$ and a basepoint $x_0$.  Clearly the subgroup generated by $V\setminus v$ is a right-angled Artin group.  If there is an edge between $v'\in V\setminus v$ and $v$ then $v^k$ preserves $X_{v'}$.  If there is no edge then $v^k$ sends $X_{v'}$ to $X_v$.  Finally, $v^k$ sends $x_0$ into $X_v$.
\end{proof}

In particular, the isomorphism classes of right-angled Artin groups which can be embedded into $\Mod_{g,p}$ are precisely those which arise from replacing a finite collection of mapping classes by sufficiently high powers and looking at the resulting right-angled Artin group.  Our observations concerning ping--pong on $\bP\mathcal{ML}(\Sigma)$ show:
\begin{cor}
If $\{f_1,\ldots,f_k\}$ are pseudo-Anosov mapping classes which generate a right-angled Artin group, then this group is free.
\end{cor}

We first prove two stable, easier embedding theorems for right-angled Artin groups in $\Sigma_g$:
\begin{prop}
Let $\gam$ be a finite graph and $A(\gam)$ the associated right-angled Artin group.  Then there exists a closed surface $\Sigma=\Sigma_{g,p}$ such that $A(\gam)<\Mod_{g,p}$, and each vertex generator of $A(\gam)$ sits in $\Mod_{g,p}$ as a power of a Dehn twist about simple closed curve.
\end{prop}

\begin{proof}
Fix a surface $\Sigma$ of genus $g\geq 1$.  Since the curve complex of $\Sigma$ has infinite diameter, we may choose a collection of simple closed curves on $\Sigma$, one for each vertex of $\gam$, each of which intersects every other.  In general position, each intersection is a double intersection, so that there are no triple points.  Label the curves by the vertices of $\gam$.  If two curves $c_1,c_2$ correspond to vertices in $\gam$ which are connected by an edge then we perform the following modification of $\Sigma$:
\begin{enumerate}
\item
Find the intersection points of $c_1$ and $c_2$ and remove small neighborhoods of the intersection points.
\item
Glue in a torus with one boundary component $\partial$ in place of each intersection.
\item
Complete $c_1$ and $c_2$ to simple closed curves through each of these tori in such a way that they do not intersect.  This can be done in such a way that the homology classes of the intersections of $c_1$ and $c_2$ relative to $\partial$ are distinct.
\end{enumerate}

There are two facts to be verified now.  First that none of the curves in the modified surface $\Sigma'$ are inessential or pairwise isotopic, and secondly that any two curves whose pairwise intersections were not modified still have positive geometric intersection number.  Upon verifying these facts we apply Theorem \ref{t:main} to get the proposition.

Suppose that $c$ is a curve in $\Sigma'$.  Then $c$ has not become inessential.  Indeed, since $c$ itself was not inessential, $\Sigma\setminus c$ has at most two components, each with nontrivial topology.  The modifications done to obtain $\Sigma'$ do one of two things:
\begin{enumerate}
\item
Combine two components of $\Sigma\setminus c$ into one component when $c$ itself is modified.
\item
Increase the topological complexity of a component of $\Sigma\setminus c$.
\end{enumerate}
Thus, $c$ does not become inessential.

If $c_1,c_2$ are two curves which intersect in two points, $c_1$ and $c_2$ can be deformed by isotopy in order to intersect fewer times if and only if $c_1$ and $c_2$ cobound a bigon.  If they cobound a bigon in $\Sigma'$ then one can show that they cobound a bigon in $\Sigma$.  Finally, $c_1$ and $c_2$ are isotopic if and only if they cobound an annulus.  If they cobound an annulus in $\Sigma'$ then one can easily check that they already cobound an annulus in $\Sigma$.
\end{proof}

\begin{prop}
Let $\gam$ be a finite graph and $A(\gam)$ the associated right-angled Artin group.  Then there exists a surface $\Sigma=\Sigma_{g,p}$ such that $A(\gam)<\Mod_{g,p}$, and each vertex generator of $A(\gam)$ sits in $\Mod_{g,p}$ as a pseudo-Anosov homeomorphism on a subsurface of $\Sigma$.
\end{prop}
\begin{proof}
Let $\gam^*$ denote the complement of $\gam$, which is obtained by taking the complete graph on the vertices of $\gam$ and deleting the edges which lie in $\gam$ itself.  Find a closed surface $\Sigma_0$ of some genus large enough to accommodate $\gam^*$.  For each vertex and each edge, add a genus one handle to $\Sigma_0$ to get a surface $\Sigma$.  For each vertex $v\in\gam^*$, let $\Sigma_v$ be the connected subsurface with one boundary component which contains the handles corresponding to $v$ and all edges emanating from $v$.  If $v$ and $w$ are not adjacent, we arrange so that $\Sigma_v$ and $\Sigma_w$ are disjoint.  We then let $\psi_v$ be a pseudo-Anosov homeomorphism supported on $\Sigma_v$.  The conclusion of the proposition follows easily.
\end{proof}

One can easily modify the previous proof to arrange for $\Sigma$ to be closed.  The following is now immediate:
\begin{cor}
For each graph $\gam$, there is an embedding $A(\gam)\to\Mod_{g,p}$, where the image of the embedding is contained in the Torelli subgroup of $\Mod_{g,p}$.
\end{cor}
\begin{proof}
In the proof of the previous proposition, each $\Sigma_v$ can be arranged to have genus at least two, in which case we can choose $\psi_v$ to sit in the Torelli group.
\end{proof}

The following corollary was suggested in the introduction, and it is trivial to establish now.
\begin{cor}
Let $G$ be enveloped by a right-angled Artin group.  Then there is a surface $\Sigma$ such that $G<\Mod(\Sigma)$.
\end{cor}

We now wish to fix a genus and embed various right-angled Artin groups into $\Mod_{g,p}$.  Note that by the main result of \cite{BLM}, there is a rank obstruction on embedding a given right-angled Artin group in a given mapping class group, a fact which is reflected in the structure of the curve complex.  This is an example of a local obstruction to embedding a right-angled Artin group, since the obstruction can be determined by understanding the neighborhood of radius $1$ about each vertex in $\gam$.

We will now give a slightly less obvious example of a local obstruction to embedability of right-angled Artin groups in mapping class groups.  Let $F$ be a nonabelian free group of finite rank.  Note that it is easy to embed a copy of $F\times F$ inside of $\Mod_g$ for $g\geq 2$.  Now let $P_n=F\times\cdots\times F$ be a finite product of free groups, which we realize as a right-angled Artin group on a graph.

\begin{prop}
Suppose $P_n<\Mod_g$.  Then $n\leq 3(g-1)/2$ when $g$ is odd and $n\leq 3g/2$ when $g$ is even.  These bounds are sharp.
\end{prop}
\begin{proof}
First suppose that $P_n$ is generated by pseudo-Anosov homeomorphisms supported on subsurfaces, and that these are the vertex generators of $P_n$.  A subsurface which supports a pseudo-Anosov homeomorphism must either be a torus with one boundary component or have Euler characteristic at most $-2$.  To realize $P_n$ in the way we wish, we must partition $\Sigma$ into $n$ subsurfaces of this type.  The Euler characteristic of $\Sigma_g$ is $2-2g$, so that if we have $n_T$ tori and $n_S$ other surfaces, we have \[n_S\leq \frac{2g-2-n_T}{2}.\]  Since there are at most $g$ tori, $g/2-1\leq n_S\leq g-1$.  Every time we increase the number of tori by one, $n_S$ decreases by $1/2$.  The bound follows for pseudo-Anosov homeomorphisms supported on subsurfaces.

Next, suppose that $P_n$ is generated by twists about multicurves, and that these are the vertex generators of $P_n$.  Choose a multicurve out of each factor of $P_n$.  Clearly there are no more than $3g-3$ of these multitwists, since we may combine them all and extend to a pants decomposition of $\Sigma_g$.  Again, the least complicated surfaces which support two non-commuting Dehn twists are the once punctured torus and the four--times punctured sphere.  Since factors of $P_n$ commute with each other, we may separate them from each other, in the sense that all the curves which generate twists in different factors are supported on disjoint subsurfaces.  The same combinatorial argument as above gives the upper bound.

In the general case, we may assume that each vertex generator of $P_n$ is either a pseudo-Anosov homeomorphism supported on a connected subsurface or a Dehn twist about a simple closed curve, since otherwise we can extract smaller subsurfaces supporting direct factors of a product of free groups $P_n$.  Two non-commuting vertex generators have a pair of non-commuting twists as factors, or one of generators has a pseudo-Anosov component.  By the same combinatorial arguments as above, we must be able to build $\Sigma_g$ out of $n$ once punctured tori and other surfaces of Euler characteristic no more than $-2$, so that the desired bound holds.

The sharpness of the bound can be achieved easily by decomposing the surface into $g$ tori with one boundary component and $g/2$ four--times punctured spheres, when $g$ is even.  When $g$ is odd, the right bound is $(g-1)/2$ four--times punctured spheres.
\end{proof}

From the work of \cite{CW} it seems likely that there are more global obstructions to embedding particular right-angled Artin groups into mapping class groups, and this is a potential direction for further inquiry.  We can say a few basic things:

Recall the notion of the {\bf curve complex} $\mC(\Sigma)$.  This is the a simplicial flag complex whose vertices are isotopy classes of essential, nonperipheral, simple closed curves in $\Sigma$ and whose incidence relation is given by simultaneous disjoint realization.  It is easy to see that $\mC(\Sigma)$ is locally infinite and finite dimensional.  We will mostly be concerned with the $1$--skeleton of $\mC(\Sigma)$, which we also denote by $\mC(\Sigma)$.

Let $\gam,\Delta$ be graphs and let $f:\gam\to\Delta$ be an embedding.  We call $f$ a {\bf flag embedding} if when two vertices of $\gam$ are sent to adjacent vertices of $\Delta$, they were adjacent in $\gam$.  Equivalently, $f$ preserves the adjacency relation, or equivalently $f(\Delta)$ is an induced subgraph of $\gam$ which is isomorphic to $\Delta$.

\begin{prop}
Let $A(\gam)$ be a right-angled Artin group.  Then $A(\gam)$ can be embedded in $\Mod_{g,p}$ in such a way that the right-angled Artin system determined by $\gam$ is mapped to powers of twists about simple closed curves if and only if there exists a flag embedding $f:\gam\to\mC(\Sigma)$.
\end{prop}
\begin{proof}
Suppose that $f:\gam\to\mC(\Sigma)$ is a flag embedding.  The images of two vertices of $\gam$ are adjacent if and only if their images were adjacent in $\mC(\Sigma)$.  It is well--known that two twists commute if and only if they are about disjointly realized curves, and sufficiently large positive powers of any finite collection of twists generate a right-angled Artin group by Theorem \ref{t:main}.  Two curves in $\Sigma$ are disjointly realized if and only if the corresponding vertices in $\mC(\Sigma)$ are adjacent.

Conversely, let $A(\gam)\to\Mod_{g,p}$ be an embedding as in the statement of the lemma.  Identify the curve about which each vertex generator of $A(\gam)$ is a power of a twist, and identify the corresponding vertices in the curve complex.  Filling in the edges between vertices which are adjacent in $\mC(\Sigma)$ produces a copy of $\gam$ in $\mC(\Sigma)$.  It follows that the obvious map $\gam\to\mC(\Sigma)$ is a flag embedding.
\end{proof}

\section{Embeddings and cohomology}
\subsection{Commensurability questions}
We first provide a proof that mapping class groups of genus $g\geq 3$ surfaces cannot be commensurable with right-angled Artin groups, using soft cohomological dimension arguments:

\begin{prop}
If $g>2$ or if $g=2$ and $p>0$ then $\Mod_{g,p}$ is not commensurable with a right--angled Artin group.
\end{prop}
\begin{proof}
Let $\gam$ be a right-angled-Artin group such that $\Mod_{g,p}$ is commensurable with $\gam$.  Then $\Mod_{g,p}$ and $\gam$ have the same virtual cohomological dimension.  It is a standard fact about right-angled Artin groups that their cohomological dimension is equal to the rank of a maximal abelian subgroup (equivalently this is the dimension of the Salvetti complex $\mathcal{S}(\gam)$, which is a finite-dimensional $K(\gam,1)$).  See \cite{C}, for instance.  The work of Harer in \cite{H} shows that if $\Sigma$ is closed of genus $g$ then $\vcd\Mod_{g}=4g-5$.  On the other hand, the maximal abelian subgroup of $\Mod_{g}$ has rank $3g-3$, so $\vcd\gam=3g-3$.  The only way $3g-3=4g-5$ is when $g=2$.

If $p\neq 0$ then \cite{H} shows that $\vcd\Mod_{g,p}=4g+p-4$.  The rank of a maximal abelian subgroup of $\Mod_{g,p}$ is $3g-3+p$.  These two numbers are equal only when $g=1$.
\end{proof}

\subsection{Virtual embedding questions}
The question of whether a mapping class group virtually injects into right-angled Artin group is more subtle.  Let $\Mod^1(\Sigma)$ denote the mapping class group of $\Sigma$ with one marked point.  We have that $\Mod^1(\Sigma)$ contains a copy of $\pi_1(\Sigma)$.

\begin{lemma}\label{l:mappingtorus}
Let $\Sigma$ be closed and let $H<\Mod^1(\Sigma)$ be a finite index subgroup which contains $\pi_1(\Sigma)$.  Then $H$ is not residually torsion-free nilpotent.  In particular, $H$ does not inject into a right-angled Artin group.
\end{lemma}
\begin{proof}
$H$ virtually contains every group which is a semidirect product of the form \[1\to\pi_1(\Sigma)\to \gam\to \bZ\to 1.\]  In particular, $H$ virtually contains the fundamental group of each fibered hyperbolic manifold which admits a fibration with fiber $\Sigma$.  It is possible to find elements of the symplectic group which act irreducibly on $H_1(\Sigma,\bQ)$, and all of whose powers act irreducibly on $H_1(\Sigma,\bQ)$.  It follows that $\gam^{ab}$ is cyclic.  Since nilpotent groups with cyclic abelianization are themselves cyclic, we have that $\gam$ is not residually torsion--free nilpotent.
\end{proof}

Modifying the argument of Lemma \ref{l:mappingtorus} to show that mapping class groups are not virtually special seems difficult (see Propositions \ref{l:one} and \ref{l:two} below).  To overcome these difficulties, we exploit certain non-virtually special subgroups of sufficiently complicated mapping class groups.  Recall that a group $G$ is called {\bf residually finite rationally solvable} or {\bf RFRS} if there exists an exhaustive filtration of normal (in $G$) finite index subgroups of $G$ \[G=G_0>G_1>G_2>\cdots\] such that $G_{i+1}>ker\{G_i\to G_i^{ab}\otimes\bQ\}$.  Agol proves that right-angled Artin groups are virtually RFRS (see \cite{A}).  Furthermore, any subgroup of a RFRS group is RFRS again, as follows easily from the definition.

\begin{lemma}
Let $\Sigma$ be a surface of genus $g\geq 2$ and one boundary component, and let $H<\Mod(\Sigma)$ be a finite index subgroup.  Then $H$ is not RFRS.  In particular, $H$ does not inject into a right-angled Artin group.
\end{lemma}
\begin{proof}
We have that $\Mod(\Sigma)$ contains a copy of the fundamental group of the unit tangent bundle $U$ of a genus two surface $\Sigma_2$ (see \cite{FM}, though this fact is essentially due to J. Birman).  We have that $\pi_1(U)$ fits into a non-split central extension \[1\to\bZ\to\pi_1(U)\to\pi_1(\Sigma_2)\to 1.\]

Let $\gam$ be the fundamental group of a nontrivial circle bundle of a closed surface $\Sigma$ with trivial monodromy, and
let $\{G_i\}$ be any candidate for an exhausting sequence which witnesses the RFRS condition on $\gam$.  We will again show that $\{G_i\}$ cannot exhaust $\gam$.  We view $\gam$ as a central extension \[1\to\bZ\to\gam\to\pi_1(\Sigma)\to 1.\]  Consider $\gam^{ab}$.  Clearly there is a surjection $\gam^{ab}\to H_1(\Sigma,\bZ)$.  Suppose that the central copy of $\bZ$ in $\gam$ maps injectively to $\gam^{ab}$.  We would then be able to describe $\gam^{ab}$ as a central extension of the form \[1\to\bZ\to \gam^{ab}\to H_1(\Sigma,\bZ)\to 1,\] and since $\gam^{ab}$ is abelian the extension would have to split.  In particular, its classifying cocycle in $H^2(\bZ^{2g},\bZ)$ would be trivial, where $g$ is the genus of $\Sigma$.  The pullback cocycle in $H^2(\Sigma,\bZ)$ would also be zero, so that the extension describing $\gam$ itself would have to split, a contradiction.  It follows that the central copy of $\bZ$ is in the kernel of the map $\gam\to\gam^{ab}\otimes\bQ$.  We therefore have $\bZ<G_1$.

We inductively suppose that $\bZ<G_i$. Write $H_i=G_i/\bZ$, where we quotient by the central copy of $\bZ$.  Since $H_i$ is a subgroup of $\pi_1(\Sigma)$, we get a restriction map $H^2(\pi_1(\Sigma),\bZ)\to H^2(H_i,\bZ)$.  Recall that since $H_i$ has finite index in $\pi_1(\Sigma)$, we have that by general cohomology of groups there is a natural corestriction map $H^2(H_i,\bZ)\to H^2(\pi_1(\Sigma),\bZ)$, and the composition of the corestriction with the restriction (often written $Cor\circ Res$) is multiplication by $[\pi_1(\Sigma):H_i]$ (cf. \cite{Br}).  It follows that if $e\neq 0$ then $Cor\circ Res(e)\neq 0$ so that $Res(e)\neq 0$.  As in the base case, it follows that the central copy of $\bZ$ cannot map injectively into $G_i^{ab}$, so that this copy of $\bZ$ is in the kernel of the map $G_i\to G_i^{ab}\otimes\bQ$.  It follows that $\bZ<G_{i+1}$.  It follows that the original central copy of $\bZ$ was contained in each $G_i$, which is what we set out to prove.
\end{proof}

Alternatively, if $\pi_1(U)$ was virtually RFRS then $U$ would virtually fiber over the circle, which cannot happen.  In some sense, the failure of the fundamental group of the unit tangent bundle of a surface to be RFRS is the algebraic phenomenon underlying the failure of the unit tangent bundle to fiber over the circle.

As a corollary to the considerations above, we can obtain Theorem \ref{t:notvirt}:
\begin{proof}[Proof of Theorem \ref{t:notvirt}]
If $\Sigma$ contains a genus two surface with one boundary component $S$ as a subsurface then its mapping class group contains the mapping class group of $S$.  In particular, $\Mod(\Sigma)$ contains $\pi_1(U)$, the fundamental group of the unit tangent bundle of $\Sigma_2$.  But we have shown that such a mapping class group cannot virtually inject into a right-angled Artin group.
\end{proof}

\subsection{Virtual speciality questions}
It is interesting to consider the cases of simpler mapping class groups for which our methods here do not work.  If a given mapping class group is virtually special, then it must be virtually residually torsion--free nilpotent.  Consider the following basic observations, which can be viewed as variations on Lemma \ref{l:mappingtorus}:

\begin{prop}\label{l:one}
Let $\phi\in\Aut(F_n)$ be an automorphism and let \[1\to F_n\to N_{\phi}\to\bZ\to 1\] be the semidirect product formed by taking the $\bZ$--conjugation given by $\phi$.  If the action of $\phi$ has no fixed vector in $H_1(F_n^{ab},\bQ)$, then $N_{\phi}$ is not residually torsion--free nilpotent.
\end{prop}
\begin{proof}
The abelianization of $N_{\phi}$ factors through the semidirect product \[1\to F_n^{ab}\to \overline{N_{\phi}}\to \bZ\to 1.\]  Since $\phi$ has no fixed vector in $H_1(F_n^{ab},\bQ)$ then $N_{\phi}^{ab}\otimes\bQ$ has rank one.  It follows that any torsion--free nilpotent quotient of $N_{\phi}$ has cyclic abelianization.  It is well--known that a nilpotent group with cyclic abelianization is itself cyclic, whence the claim.
\end{proof}

\begin{prop}\label{l:two}
Let $\gam<\Aut(F_n)$ be a subgroup which contains $\Inn(F_n)$ and let $\phi\in \gam$.  Then $N_{\phi}<\gam$.
\end{prop}
\begin{proof}
Let $g\in F_n$ and let $i_g$ denote conjugation by $g$.  We have that $\phi\circ i_g\circ\phi^{-1}=i_{\phi(g)}$.  Indeed, \[\phi\circ i_g\circ\phi^{-1}(h)=\phi(g\phi^{-1}(h)g^{-1})=\phi(g)h\phi(g)^{-1}.\]  It follows that the subgroup of $\gam$ generated by $\phi$ and $F_n$ is isomorphic to $N_{\phi}$.
\end{proof}

Both of these propositions hold with automorphisms of free groups replaced by mapping classes.
Consider the set of hyperbolic $3$-manifolds built by suspending pseudo-Anosov homeomorphisms in $\Mod_{g,p}$ with fundamental groups $\{\gam_i\}$.  Here, the meaning of $\Mod_{g,p}$ is that we require the homeomorphisms and isotopies to preserve the punctures of $\Sigma_{g,p-1}$, together with an extra marked point.  This way, we get an action of $\Mod_{g,p}$ on $\pi_1(\Sigma_{g,p-1})$.  In particular, whenever $\psi\in\Mod_{g,p}$, we have that the fundamental group $\pi_1(M_{\psi})$ of the mapping torus of $\psi$ is contained in $\Mod_{g,p}$.

If $\Mod_{g,p}$ is virtually special then it follows that there is a universal finite index to which one could pass in any of the $\{\gam_i\}$ and obtain first Betti number greater than one.  Indeed, suppose that $G<\Mod_{g,p}$ embeds into a right-angled Artin group.  Then $G\cap\gam_i$ is residually torsion--free nilpotent, so that $b_1(G\cap \gam_i)\geq 2$.

This observation seems strange at first, but consider why it holds for a once--punctured torus.  Taking any nontrivial abelian cover, there is a finite index subgroup of the mapping class group which preserves the homology class of each lift of the puncture, so that there is a lot of invariant homology in any suspension of any mapping class.

We close with the following conjectures:

\begin{conje}
The following groups are virtually special:
\begin{enumerate}
\item
The mapping class group $\Mod_{2}$.
\item
The braid groups $B_n$, for all $n$.
\end{enumerate}
\end{conje}


\begin{thebibliography}{99}
\bibitem{A}Ian Agol.  Criteria for virtual fibering.  {\it J. Topol.}, 1, 269--284, 2008.
\bibitem{BaumRose}G. Baumslag and J.E. Roseblade. Subgroups of direct products of free groups, {\it J. London Math. Soc.} 30, 44--52, 1984.
\bibitem{BehrChar}Jason Behrstock and Ruth Charney.  Divergence and quasi--morphisms of right-angled Artin groups.  Preprint.
\bibitem{BP}Ricardo Benedetti and Carlo Petronio.  {\it Lectures on hyperbolic geometry}.  Universitext, Springer, 1992.
\bibitem{BeBr}Mladen Bestvina and Noel Brady.  Morse theory and finiteness properties of groups.  {\it Invent. Math.} 129, 123--139, 1997.
\bibitem{BLM}Joan S. Birman and Alex Lubotzky and John McCarthy.  Abelian and solvable subgroups of the mapping class groups.  {\it Duke Math. J.} 50, no. 4, 1107--1120, 1983.
\bibitem{Brid1}Martin R. Bridson.  Semisimple actions of mapping class groups on $\CAT(0)$ spaces.  To appear in {\it The Geometry of Riemann Surfaces}, LMS Lecture Notes 368.
\bibitem{Brid2}Martin R. Bridson.  On the dimension of $\CAT(0)$ spaces where mapping class groups act.  Preprint, 2009.
\bibitem{BridMill}Martin R. Bridson and Charles F. Miller III.  Structure and finiteness properties of subdirect products of groups.  To appear in {\it Proc. London Math. Soc.}
\bibitem{Br}Kenneth S. Brown.  {\it Cohomology of groups}.  Graduate Texts in Mathematics, no. 87, Springer, New York, 1982.
\bibitem{C}Ruth Charney.  An introduction to right-angled Artin groups.  {\it Geom. Dedicata}, 125, 141--158, 2007.
\bibitem{ChFar}Ruth Charney and Michael Farber.  Random groups arising as graph products.  Preprint.
\bibitem{CLM}Matt Clay, Chris Leininger and Johanna Mangahas.  The geometry of right angled Artin subgroups of mapping class groups.  To appear in {\it Groups, Geom. Dyn.}
\bibitem{CF}John Crisp and Benson Farb.  The prevalence of surface groups in mapping class groups.  Preprint.
\bibitem{CP}John Crisp and Luis Paris.  The solution to a conjecture of Tits on the subgroup generated by the squares of the generators of an Artin group.  {\it Invent. Math.} 145, no. 1, 19--36, 2001.
\bibitem{CW}John Crisp and Bert Wiest.  Quasi-isometrically embedded subgroups of braid and diffeomorphism groups.  {\it Trans. Amer. Math. Soc.} 359, no. 11, 5485--5503, 2007.
\bibitem{Day}Matthew B. Day.  Peak reduction and finite presentations for automorphism groups of right-angled Artin groups.  {\it Geom. Topol.} 13, 817--855, 2009.
\bibitem{dlH}Pierre de la Harpe.  {\it Topics in Geometric Group Theory}.  {\it University of Chicago Press}, 2000.
\bibitem{Droms1}Carl Droms.  Subgroups of graph groups.  {\it J. Algebra} 110, 519--522, 1987.
\bibitem{Droms2}Carl Droms.  Isomorphisms of graph groups.  {\it Proc. Amer. Math. Soc.} 100, no. 3, 407--408, 1987.
\bibitem{Farb}Benson Farb.  Some problems on mapping class groups and moduli space.  In {\it Problems on Mapping Class Groups and Related Topics}.  {\it Proc. Symp. Pure and Applied Math.}, Volume 74, 2006.
\bibitem{FLM}Benson Farb and Alexander Lubotzky and Yair Minsky.  Rank one phenomena for mapping class groups.  {\it Duke Math. J.}, Vol. 106, No. 3, 581--597, 2001.
\bibitem{FM}Benson Farb and Dan Margalit.  {\it A Primer on the mapping class group}.  {\it Princeton Mathematical Series}, 49. {\it Princeton University Press, Princeton, NJ}, 2012.
\bibitem{FLP}A. Fathi and F. Laudenbach and V. Po\'enaru.  {\it Travaux de Thurston sur les Surfaces}.  Soc. Math. de France, Paris,  Ast\'erisque 66-67,  1979.
\bibitem{FP}Edward Formanek and Claudio Procesi.  The automorphism group of a free group is not linear.  {\it J. Algebra} 149, no. 2, 494--499, 1992.
\bibitem{Fujiwara}Koji Fujiwara.  Subgroups generated by two pseudo-Anosov elements in a mapping class group. I. Uniform exponential growth. {\it Groups of diffeomorphisms}, 283--296, {\it Adv. Stud. Pure Math.}, 52, {\it Math. Soc. Japan, Tokyo}, 2008.
\bibitem{Fun}Louis Funar.  On power subgroups of mapping class groups.  Preprint.
\bibitem{Gold}William M. Goldman.  {\it Complex hyperbolic geometry}.  Oxford Mathematical Monographs, 1999.
\bibitem{Gub}Joseph Gubeladze. The isomorphism problem for commutative monoid rings.  {\it J. Pure Appl. Algebra}, 129(1):35--65, 1998.
\bibitem{HW}Fr\'ed\'eric Haglund and Daniel T. Wise.  Special cube complexes. {\it GAFA, Geom. funct. anal.} 17, 1551--1620, 2008.
\bibitem{Ham}Ursula Hamenst\"adt.  Geometric properties of the mapping class group.  In {\it Problems on Mapping Class Groups and Related Topics}.  {\it Proc. Symp. Pure and Applied Math.}, Volume 74, 2006.
\bibitem{H}John L. Harer.  The virtual cohomological dimension of the mapping class group of an orientable surface.  {\it Invent. Math.} 84, no. 1, 157--176, 1986.
\bibitem{I}Nikolai Ivanov.  Subgroups of Teichm\"uller modular groups.  {\it Translations of Mathematical Monographs}, 115, 1992.
\bibitem{JS}Tadeusz Januszkiewicz and Jacek \'Swiatkowski.  Hyperbolic Coxeter groups of large dimension.  {\it Comment. Math. Helv.} 78, 555--583, 2003.
\bibitem{KL}Michael Kapovich and Bernhard Leeb.  Actions of discrete groups on nonpositively curved spaces.  {\it Math. Ann.} 306, no. 2, 341--352, 1996.
\bibitem{Ka}Svetlana Katok.  {\it Fuchsian groups}.  Chicago Lectures in Mathematics. {\it University of Chicago Press}, Chicago, IL, 1992.
\bibitem{KK}Sang-hyun Kim and Thomas Koberda.  Embedability between right-angled Artin groups.  Preprint 2011.
\bibitem{KobIMS}Thomas Koberda.  Ping--pong lemmas with applications to geometry and topology.  To appear in {\it IMS Lecture Notes}, Singapore, 2012.
\bibitem{Lau}Michael R. Laurence.  A generating set for the automorphism group of a graph group.  {\it J. London Math. Soc.} (2) 52 no. 2, 318--334, 1995.
\bibitem{Lev}Gilbert Levitt.  Unsolvability of the isomorphism problem for [free abelian]--by--free groups.  Preprint.
\bibitem{Loen}Michael L\"onne.  Presentations of subgroups of the braid group generated by powers of band generators.  {\it Topology Appl.} 157, no. 7, 1127--1135, 2010.
\bibitem{LR}Chris Leininger and Alan Reid.  A combination theorem for Veech subgroups of the Mapping class group. {\it Geom. and Funct. Anal.} 16, 403--436, 2006.
\bibitem{LySch}Roger C. Lyndon and Paul E. Schupp.  {\it Combinatorial group theory}.  Springer, New York, 1977.
\bibitem{Mi}K.A. Mihailova.  The occurrence problem for free products of groups. (Russian) {\it Mat. Sb. (N.S.)} 75 (117) 199--210, 1968.
\bibitem{MP}John McCarthy and Athanase Papadopoulos.  Dynamics on Thurston's sphere of projective measured foliations.  {\it Comment. Math. Helv.} 64, no. 1, 133--166, 1989.
\bibitem{Mil}Charles F. Miller, III.  {\it On group-theoretic decision problems and their classification}.  {\it Ann. of Math. Studies}, 68, 1971.
\bibitem{Mil2}Chales F. Miller III.  Decision problems for groupsÑsurvey and reflections, in {\it Algorithms and classification in combinatorial group theory} (Berkeley, CA, 1989), {\it Math. Sci. Res. Inst. Publ.}, 23, Springer--Verlag, 1--59, 1992.
\bibitem{HeeOh}Hee Oh.  Discrete subgroups generated by lattices in opposite horospherical subgroups.  {\it J. Algebra} 203, 621--676, 1998.
\bibitem{PH}R.C. Penner and J.L. Harer.  {\it Combinatorics of train tracks.}  Annals of Mathematics Studies, no. 125, {\it Princeton University Press}, 1992.
\bibitem{Sab}Lucas Sabalka.  Embedding right-angled Artin groups into graph braid groups.  {\it Geometriae Dedicata} 124, no. 1, 191--198, 2007.
\bibitem{Sab1}Lucas Sabalka.  On rigidity and the isomorphism problem for tree braid groups.  {\it Groups, Geometry, and Dynamics}, 3(3):469--523, 2009.
\bibitem{Ser}H. Servatius.  Automorphisms of graph groups. {\it J. Algebra} 126, no. 1, 34--60, 1989.
\bibitem{St}John Stallings.  The topology of finite graphs.  {\it Invent. Math.} 71, no. 3, 551--565, 1983.
\bibitem{Still}John Stillwell.  The occurrence problem for mapping class groups.  {\it Proc. Amer. Math. Soc.} 101, no. 3, 411--416, 1987.
\bibitem{Tits}Jacques Tits.  Free subgroups in linear groups.  {\it J. Algebra}, 20, 250--270, 1972.
\bibitem{Wang}Stephen Wang.  Representations of surface groups and right-angled Artin groups in higher rank. {\it Algebr. Geom. Topol.} 7, 1099--1117, 2007.
\end{thebibliography}
\end{document}